\documentclass[runningheads]{llncs}
\usepackage[colorlinks]{hyperref}
\usepackage{graphicx}
\usepackage{amsmath,amssymb}
\usepackage{mathrsfs}           
\usepackage{mathtools}          
\usepackage{stmaryrd}           
\usepackage{mathpartir}
\usepackage{todonotes}
\usepackage{relsize}
\usepackage{macros}
\allowdisplaybreaks

\usepackage{chngcntr}
\usepackage{apptools}
\AtAppendix{\counterwithin{thm}{section}}

\usepackage{amsthm}

\usepackage{cleveref}
\crefname{equation}{}{}

\theoremstyle{plain}
\newtheorem{thm}{Theorem}
\newtheorem{prop}[thm]{Proposition}
\newtheorem{lem}[thm]{Lemma}
\newtheorem{cor}[thm]{Corollary}

\theoremstyle{definition}
\newtheorem{dfn}[thm]{Definition}

\begin{document}


\title{Time Warps, from Algebra to Algorithms}
%
%
\author{Sam van Gool\inst{1} \and Adrien Guatto\inst{1} \and George Metcalfe\inst{2}\thanks{Supported by Swiss National Science Foundation grant 200021$\_$165850.} \and Simon Santschi\inst{2}}
\authorrunning{S. van Gool et al.}
%
\institute{IRIF, Universit{\'e} de Paris, France\\
\email{\{guatto,vangool\}@irif.fr}
\and Mathematical Institute, University of Bern, Switzerland\\
\email{\{george.metcalfe,simon.santschi\}@math.unibe.ch}}
\maketitle


\begin{abstract}


Graded modalities have been proposed in recent work on programming languages as a general framework for refining type systems with intensional properties. In particular, continuous endomaps of the discrete time scale, or~\emph{time warps}, can be used to quantify the growth of information in the course of program execution. Time warps form a complete residuated lattice, with the residuals playing an important role in potential programming applications. In this paper, we study the algebraic structure of time warps, and prove that their equational theory is decidable, a necessary condition for their use in real-world compilers. We also describe how our universal-algebraic proof technique lends itself to a constraint-based implementation, establishing a new link between universal
algebra and verification technology.


\end{abstract}



\section{Introduction}\label{section:introduction}

Program types are almost as old as programs themselves.
Their initial role was to allow compilers to determine data sizes at compilation
time, e.g., distinguishing machine integers from double precision
numbers~\cite{Fortran-1957}.
Type system research has developed tremendously since these humble beginnings,
benefiting from close connections to logic~\cite{Howard-1980}.
For example, dependent types are expressive enough to serve as
specification languages for program results~\cite{Agda,Coq}.

Another line of research into type systems aims to classify not only~\emph{what}
programs compute, but also~\emph{how} they do so.
Such type systems describe the~\emph{effect} of a program---e.g., which
parts of memory it modifies~\cite{GiffordLucassen-1988}---or the
\emph{resources} it requires---e.g., how long it takes to
run~\cite{GhicaSmith-2014}.
Recently, \emph{graded
  modalities}~\cite{FujiiKatsumataMellies-2016,GaboardiKatsumataOrchardBreuvartUustalu-2016}
have emerged as a unified setting for describing effect- and resource-annotated
types.
A graded modality~$\Box$ allows programmers to form a new type~$\Box_f A$ from a
type~$A$ and a~\emph{grading}~$f$.
The meaning of~$\Box_f A$ depends on the system at hand, but can generally
be understood as a modification of~$A$ that includes the behavior prescribed by~$f$.

In many cases, gradings come equipped with an ordered algebraic structure
that is relevant for programming applications.
Most commonly, they form a monoid whose binary operation corresponds to a notion
of composition such that~$\Box_{gf} A$ is related to~$\Box_f \Box_g A$.
It is also often the case that gradings can be ordered by some sort
of~\emph{precision} ordering along which the graded modality acts
contravariantly.
That is, we have a~\emph{generic} program of type~$\Box_g A \to \Box_f A$ if~$f \le g$, 
allowing us to freely move from more to less precise types.
As a consequence, the structure of this ordering is reflected by the
operations available on types; for example, when the infimum of~$f$ and~$g$
exists, it permits the conversion of two values of types~$\Box_f A$ and~$\Box_g B$ into
a single value of type~$\Box_{f \mt g} (A \times B)$.

The additional flexibility and descriptive power gained by adopting graded
modalities in a programming language comes at a price, however.
The language implementation must now be able to manipulate gradings in various
ways; in particular, it should be able to decide the ordering between gradings in order to
distinguish between well-typed and ill-typed programs.
In this paper, we address this issue for a specific class of gradings known as~\emph{time warps}: sup-preserving functions on $\som = \om \cup \{ \om \}$, or, equivalently, monotonic functions $f \colon \som \to \som$ satisfying $f(0)=0$ and $f(\om) = \bigvee \{ f(n) \mid n \in \om \}$~\cite{Guatto-2018}.
Informally, time warps describe the growth of data along program execution.
In this setting, any type~$A$ describes a family of sets~$(A_n)_{n\in\om}$, where~$A_n$ is the set of values classified by~$A$ at execution step~$n$.
The type~$\Box_f A$ classifies the set of values of~$A_{f(n)}$ at step~$n$.
This typing discipline generalizes a long line of works on programming languages
for embedded systems~\cite{CaspiPouzet-1996} and type theories with modal
recursion
operators~\cite{Nakano-2000,BirkedalMogelbergSchwinghammerStovring-2012}.

Let us denote the set of time warps by $\Warp$. Then~$\langle \Warp,\circ,\id\rangle$ is a monoid, where $fg \defeq f\circ g$ denotes the composition of $f,g\in\Warp$, and $\id$ is the identity function. Moreover, equipping $\Warp$ with the pointwise order, defined by
\[
f\le g\: :\Longleftrightarrow \enspace f(p)\le g(p)\text{ for all }p\in\som,
\]
yields a complete distributive lattice $\langle \Warp,\mt,\jn\rangle$ satisfying, for all $f,g_1,g_2,h\in \Warp$, 
\[
f(g_1\jn g_2)h = fg_1h\jn fg_2h\:\text{ and }\:f(g_1\mt g_2)h = fg_1h\mt fg_2h,
\]
with a least element  $\bot$ that maps all $p\in\som$ to $0$, and a greatest element $\top$ that maps all $p\in\som{\setminus}\{0\}$ to $\om$. Note that the operation $\circ$ is a double quasi-operator on this lattice in the sense of~\cite{GP07b,GP07a}, and that the structure $\langle \Warp,\mt,\jn,\circ,\id\rangle$ belongs to the family of unital quantales of sup-preserving functions on a complete lattice studied in~\cite{San20}. 

The monoidal structure of time warps plays the expected role in programming applications. In particular,~$\Box_{gf} A$ and~$\Box_f \Box_g A$ are isomorphic, as are~$\Box_\id A$ and~$A$. However, time warps also admit further additional algebraic structure of interest for programming. Since they are sup-preserving, there exist binary operations $\ld,\rd$ on $\Warp$, called {\em residuals},  satisfying for all $f,g,h\in\Warp$,
\[
f\le h\rd g \iff fg\le h \iff g\le f\ld h.
\]
From a programming perspective, residuals play a role similar to that of weakest
preconditions in deductive verification.
The type~$\Box_{h \rd g} A$ can be seen as the~\emph{most general} type~$B$
such that~$\Box_h A$ can be sent generically to~$\Box_g B$.
Similarly,~$f \ld h$ is the most general~(largest) time warp~$f'$ such
that~$\Box_h A$ can be sent generically to~$\Box_{f'} \Box_f A$.
Such questions arise naturally when programming in a modular
way~\cite{Guatto-2018}, justifying the consideration of residuated structure in
gradings.

The algebraic structure $\WarpA = \langle \Warp,\mt,\jn,\circ,\ld,\rd,\id,\bot,\top\rangle$, referred to here as the \emph{time warp algebra}, belongs to the family of (bounded) residuated lattices, widely studied as algebraic semantics for substructural logics~\cite{BT03,GJKO07,MPT10}. The main goal of this paper is to prove the following theorem, a necessary condition for the use of time warps in real-world compilers:

\begin{thm}\label{theorem:main_result}
The equational theory of the time warp algebra $\WarpA$ is decidable.
\end{thm}
\noindent
A {\em time warp term} is a member of the term algebra over a countably infinite set of variables of the algebraic language with binary operation symbols $\mt,\jn,\circ,\ld,\rd$, and constant symbols $\id,\bot,\top$, and a {\em time warp equation} consists of an ordered pair of terms $s,t$, denoted by $s\eq t$. Let $s \le t$ denote the equation $s\mt t\eq s$, noting that $\WarpA\models s\eq t$ if, and only if, $\WarpA\models s\le t$ and $\WarpA\models t\le s$, and, by residuation, $\WarpA\models s\le t$ if, and only if, $\WarpA\models\id\le t\rd s$. Clearly, to prove \Cref{theorem:main_result}, it will suffice to provide an algorithm that decides $\WarpA\models\id\le t$ for any time warp term $t$.


\subsection*{Overview of the proof of \Cref{theorem:main_result}}

We prove \Cref{theorem:main_result} by describing an algorithm with the following behavior:
 
\begin{newlist}

\item[]		{\bf Input.}	A time warp term $t$ in the variables $x_1, \dots, x_k$.

\item[]		{\bf Output.} If $\WarpA \models id \leq t$, the algorithm returns `{\tt Valid}'; if $\WarpA \not\models \id \leq t$, the algorithm returns `{\tt Invalid} at $(\hat{f}_1,\dots,\hat{f}_k,p)$' for some $p \in \som$ and finite descriptions $\hat{f}_1,\dots,\hat{f}_k$ of time warps $f_1,\dots,f_k$, such that $\sem{t}(p) < p$, where $\sem{t}$ is the time warp obtained from $t$ by mapping each $x_i$ to $f_i$.
\end{newlist}

\noindent
We now give a high-level overview of the three main steps of the algorithm; the details and the proof of its correctness will occupy us for the rest of the paper. 


\subsubsection*{I. Pre-processing into a disjunction of basic terms.}

In \Cref{section:normalform}, we show how to effectively obtain for any time warp term $t$, a time warp term
\[
t' \defeq \bigwedge_{i=1}^m \bigvee_{j=1}^{n_i} t_{i,j},
\]
such that $\WarpA \models t \approx t'$, where each $t_{i,j}$ is a \emph{basic term}, constructed using $\circ$, $\id$, $\bot$, and the defined operations $\rl{s}\defeq \id\rd s$, $\rr{s}\defeq s\ld\id$, and $\ro{s}\defeq \top\ld s$  (Theorem~\ref{theorem:simpleform}). Since  $\WarpA \models \id \leq t$ if, and only if, $\WarpA \models \id \leq \bigvee_{j=1}^{n_i} t_{i,j}$ for each $i\in\{1,\dots,m\}$, our task is reduced to giving an algorithm with the required behavior for terms of the form $t_1 \jn \cdots \jn t_n$, where each $t_i$ is a basic term. Once we have an algorithm that solves this case, we can run it for each of the $m$ conjuncts of $t'$ in turn, returning `{\tt Invalid} at $(\hat{f}_1,\dots,\hat{f}_k,p)$' whenever this is the result of one of these runs, and otherwise `{\tt Valid}'. 


\subsubsection*{II. Finitary characterization through diagrams.} The crucial step in our algorithm is the finitary characterization of `potential counterexamples' for an equation of the form $\id \leq t_1 \jn \cdots \jn t_n$, where each $t_i$ is a basic term. Our main tool for providing these finitary characterizations is the notion of a {\em diagram}.\footnote{The name `diagram' recalls a similar concept used to prove the decidability of the equational theory of lattice-ordered groups in~\cite{HM79}.} 

Let us give an example to illustrate the basic idea. To falsify the equation $\id\le xy\rl{x} \jn \rl{y}$ in $\WarpA$, it suffices to find time warps $f_x$ and $f_y$, and an element $p \in \som$, such that $(f_x \circ f_y \circ\rl{f}_x)(p)<p$ and $\rl{f}_y(p)<p$. Although time warps are, as functions on $\som$, infinite objects, only \emph{finitely many} of the values of $f_x$ and $f_y$ are relevant for falsifying the equation. Moreover, an upper bound for the number of values required for such a counterexample can be computed. The condition  $(f_x \circ f_y \circ\rl{f}_x)(p)<p$ is `unravelled'  by stating that there exist $\alpha_1, \alpha_2, \alpha_3\in\som$  such that $\alpha_3 < p$, where $\alpha_1 \defeq \rl{f}_x(p)$, $\alpha_2 \defeq f_y(\alpha_1)$, and $\alpha_3 \defeq f_x(\alpha_2)$. More formally, using a `time variable' $\kappa$ to refer to the value $p$, we build a finite {\em sample set} $\Gamma_1 \supseteq\{\kappa, \rl{x}[\kappa], y[\rl{x}[\kappa]], x[y[\rl{x}[\kappa]]]\}$, where $\Gamma_1$ is `saturated' with extra conditions used to describe, e.g., the behavior of $\rl{f}_x$ at relevant values. Similarly, we obtain a finite saturated sample set $\Gamma_2 \supseteq\{\kappa,\rl{y}[\kappa]\}$ for the condition  $\rl{f}_y(p)<p$.  The problem of deciding if there exists a counterexample to  $\id\le xy\rl{x} \jn \rl{y}$ then becomes the problem of deciding if there exists a suitable function $\delta \colon \Gamma_1\cup\Gamma_2 \to \som$ satisfying $\delta(x[y[\rl{x}[\kappa]]]) < \delta(\kappa)$ and $\delta(\rl{y}[\kappa]]]) < \delta(\kappa)$. In particular, $\delta$ should determine partial sup-preserving functions $\hat{f}_x$ and $\hat{f}_y$ on $\som$ satisfying $\hat{f}_x(\delta(\alpha)) = \delta(x[\alpha])$ for all $x[\alpha]\in\Gamma_1\cup\Gamma_2$, and $\hat{f}_y(\delta(\alpha)) = \delta(y[\alpha])$ for all $y[\alpha]\in\Gamma_1\cup\Gamma_2$.

Clearly, not every function $\delta$ from a saturated sample set to $\som$ extends to a valuation in $\WarpA$; e.g., if $\delta(\kappa) = 0$, then we must also have $\delta(x[\kappa]) = 0$. Moreover, although time warp equations in the residual-free language can be decided by considering an algebra of sup-preserving functions on a {\em finite} totally ordered set, this is not the case for the full language.\footnote{Indeed, the equational theory of the time warp algebra without residuals coincides with the equational theory of distributive lattice-ordered monoids~\cite{CGMS21}, but an elegant (finite) axiomatization of the equational theory in the full language is not known.} \Cref{section:diagrams} develops a general theory that precisely characterizes the functions---called {\em diagrams}---that extend to valuations and can be used to falsify a given equation.  This allows us to prove that there exists a counterexample to $\id \leq t_1 \jn \cdots \jn t_n$ if, and only if, there exists a diagram $\delta\colon\Gamma\to\som$ satisfying $\delta(\kappa)>\delta(t_i[\kappa])$ for each $i\in\{1,\dots,n\}$, where $\Gamma$ is the finite saturated sample set extending $\{t_1[\kappa],\dots,t_n[\kappa]\}$ (\Cref{theorem:main-diagram-theorem}). 


\subsubsection*{III. Encoding as a satisfiability query.} In the last step of the algorithm, described in \Cref{section:logic}, we use the decidability of the satisfiability problem in the first-order logic of natural numbers with the natural ordering and successor. More precisely, we show that the existence of a diagram in Theorem~\ref{theorem:main-diagram-theorem} can be encoded as an existential first-order sentence in that signature. Concretely, our algorithm constructs a quantifier-free formula which is satisfiable in the structure $(\mathbb N, \leq, S,0)$ if, and only if, there exists a diagram as specified by Theorem~\ref{theorem:main-diagram-theorem}. Moreover, a satisfying assignment can be converted into a valuation into $\WarpA$  that provides a counterexample to the equation $\id \leq t_1 \jn \cdots \jn t_n$.





\section{A normal form for time warps}\label{section:normalform}

The main aim of this section is to provide a normal form for time warp terms. Our first step is to provide a more precise description of the left and right residuals of time warps. Note that to prove that two time warps are equal, it suffices to show that they coincide on every non-zero natural number, since for any time warp $f$, it is always the case that $f(0)=0$ and $f(\om)=\bigvee\{f(n)\mid n\in\om\}$.

\begin{lem}\label{lemma:residuals}
For any time warps $f,g$ and $p\in\som$,
\begin{newlist}
\item[\rm (a)]
$(f\ld g)(p) =
\begin{cases}
0 & \text{if }p=0;\\
\bigvee\{q\in\som\mid f(q)\le g(p)\} &  \text{if }p\in\om{\setminus}\{0\};\\
\bigvee\{q\in\som\mid(\exists m\in\om)(f(q)\le g(m))\} &  \text{if }p=\om\\
\end{cases}$
\item[\rm (b)] $(g\rd f)(p) = \bigwedge g[\{q\in\som\mid p\le f(q)\}]$.
\end{newlist}
\end{lem}
\begin{proof}
(a) Let $h$ denote the function defined by cases on the right of the equation. Clearly, $h$ is monotonic and satisfies $h(0)=0$ and $h(\om)=\bigvee\{h(n)\mid n\in\om\}$, so $h$ is a time warp. Moreover, since $f$ preserves arbitrary joins, $fh\le g$, and hence $h\le f\ld g$. For the converse, just observe that for any $n\in \om{\setminus}\{0\}$, since $f((f\ld g)(n))\le g(n)$, also $(f\ld g)(n)\le h(n)$. So $h = f\ld g$.

(b) Let $h$ be the function defined by $h(p) \defeq  \bigwedge g[\{q\in\som\mid p\le f(q)\}]$. Clearly, $h$ is monotonic and satisfies $h(0)=0$ and $h(\om)=\bigvee\{h(n)\mid n\in\om\}$, so $h$ is a time warp. Moreover, $hf\le g$, and hence $h\le g\rd f$. For the converse, let $n\in \om{\setminus}\{0\}$. If $q\in\som$ satisfies $n\le f(q)$, then $(g\rd f)(n)\le(g\rd f)(f(q))\le g(q)$, and hence $(g\rd f)(n)\le \bigwedge g[\{q\in\som\mid n\le f(q)\}] = h(n)$. So $h = g\rd f$. 
\end{proof}


Next, we show that residuals of time warps distribute over joins and meets.

\begin{lem}\label{lemma:distributivity}
For any time warps $f,g,h$,
\begin{align*}
{\rm (a)} \enspace & f\ld(g\mt h) = (f\ld g)\mt (f\ld h) & {\rm (e)} \enspace & (g\mt h)\rd f = (g\rd f)\mt (h\rd f)\\
{\rm (b)} \enspace & (g\mt h)\ld f = (g\ld f)\jn (h\ld f) & {\rm (f)} \enspace & f \rd (g\mt h) = (f\rd g)\jn (f\rd h)\\
{\rm (c)} \enspace & f\ld (g\jn h) = (f\ld g)\jn (f\ld h) & {\rm (g)} \enspace & (g\jn h)\rd f = (g \rd f)\jn (h\rd f)\\
{\rm (d)} \enspace & (g\jn h) \ld f = (g\ld f)\mt (h\ld f) & {\rm (h)} \enspace & f\rd(g\jn h) = (f\rd g)\mt (f\rd h).
\end{align*}
\end{lem}
\begin{proof}
Parts (a), (d), (e), and (h) hold in any residuated lattice (see, e.g.,~\cite{BT03}). For (b), consider any $n\in\om{\setminus}\{0\}$. Using  \Cref{lemma:residuals}(a), 
\begin{align*}
((g\mt h)\ld f)(n)
& = \bigvee\{q\in\som\mid (g\mt h)(q)\le f(n)\}\\
& = \bigvee\{q\in\som\mid g(q)\le f(n)\text{ or }h(q)\le f(n)\}\\
& = \bigvee\{q\in\som\mid g(q)\le f(n)\}\jn\bigvee\{q\in\som\mid h(q)\le f(n)\}\\
& =  ((g\ld f)\jn (h\ld f))(n).
\end{align*}
For (f), consider any $n\in\om{\setminus}\{0\}$. Using \Cref{lemma:residuals}(b), 
\begin{align*}
(f \rd (g\mt h))(n)
& = \bigwedge f[\{q\in\som\mid n\le (g\mt h)(q)\}]\\
& = \bigwedge f[\{q\in\som\mid n\le g(q) \text{ and } n\le h(q)\}]\\
& = \bigwedge f[\{q\in\som\mid n\le g(q)\}] \jn  \bigwedge f[\{q\in\som\mid n\le h(q)\}]\\
& =  ((g \rd f)\jn (h\rd f))(n).
\end{align*}
Parts (c) and (g) are proved similarly. 
\end{proof}


It follows from~\Cref{lemma:distributivity} that every time warp term is equivalent to a meet of joins of terms constructed using the operations $\circ$, $\ld$, $\rd$, $\id$, $\bot$, and $\top$. However, we can take this simplification process one step further by expressing the residuals of time warps in terms of their restrictions to certain unary operations. 

\begin{dfn}\label{definition:unaryops}
For any time warp $f$, let
\[
\rl{f}\defeq \id\rd f,\quad \rr{f}\defeq f\ld\id,\enspace\text{ and }\enspace\ro{f}\defeq \top\ld f.
\]
\end{dfn}

\begin{lem}\label{lemma:reductiontounary}
For any time warps $f,g$,
\begin{newlist}
\item[\rm (a)]	$f\ld g = \rr{f}g \jn \rr{(\top f)} \jn \ro{g}$
\item[\rm (b)]	$g\rd f = g\rl{f} \jn \ro{(\rl{f})}$.
\end{newlist}
\end{lem}
\begin{proof}
For (a), note first that clearly $\rr{f}g \jn \rr{(\top f)} \jn \ro{g} \le f\ld g$. For the converse, consider any $n\in\om{\setminus}\{0\}$. If $g(n)=0$, then, by \Cref{lemma:residuals}(a),
\[
(f\ld g)(n)
= \bigvee\{q\in\som\mid f(q)\le 0\}
= \bigvee\{q\in\som\mid \top f(q)\le \id(n)\}
= \rr{(\top f)}(n).
\]
If $g(n)\in\om{\setminus}\{0\}$, then, by \Cref{lemma:residuals}(a),
\[
(f\ld g)(n)
= \bigvee\{q\in\som\mid f(q)\le g(n)\}
= \bigvee\{q\in\som\mid f(q)\le \id(g(n))\}
= \rr{f}(g(n)).
\]
Finally, if $g(n)=\om$, then, by \Cref{lemma:residuals}(a),
\[
(f\ld g)(n)
= \bigvee\{q\in\som\mid f(q)\le \om)\}
= \om
= \bigvee\{q\in\som\mid \top(q)\le\om)\}
=  \ro{g}(n).
\]
So $f\ld g = \rr{f}g \jn \rr{(\top f)} \jn \ro{g}$.

For (b), note first that clearly $g\rl{f} \jn \ro{(\rl{f})} \le g\rd f$. For the converse, consider any $n\in\om{\setminus}\{0\}$. If $\{q\in\som\mid n\le f(q)\}=\emptyset$, then, by \Cref{lemma:residuals}(b),
\[
(g\rd f)(n) = \bigwedge g[\emptyset] = \om= (\ro{(\rl{f})})(n).
\]
Otherwise, $\{q\in\som\mid n\le f(q)\}\neq\emptyset$ and, by \Cref{lemma:residuals}(b),
\[
(g\rd f)(n)
= \bigwedge g[\{q\in\som\mid n\le f(q)\}]
= g(\bigwedge \id[\{q\in\som\mid n\le f(q)\}])
= (g\rl{f})(n).
\]
So $g\rd f = g\rl{f} \jn \ro{(\rl{f})}$. 
\end{proof}

To gain a better understanding of these defined unary operations, we observe that \Cref{lemma:residuals} yields for any $n\in\om{\setminus}\{0\}$,
\begin{align*}
\ro{f}(n) & = \max\{m\in\som\mid \omega\le f(n)\}\\
\rr{f}(n)  & = \max\{m\in\som\mid f(m)\le n\}\\
\rl{f}(n)  & = \bigwedge \{m\in\som\mid n\le f(m)\}.
\end{align*}
The following lemmas collect some simple consequences of these observations.

\begin{lem} \label{lemma:op-o}
For any time warp $f$ and $n\in \om{\setminus}\{0\}$,
\begin{align*}
\ro{f}(n) = 0 &\iff f(n) < \om \\
\ro{f}(n) = \om &\iff f(n) = \om \\
\ro{f}(\om) = 0 &\iff f(k) < \om \text{ for all } k\in \om \\
\ro{f}(\om) = \om &\iff  f(k) = \om \text{ for some } k\in \om.
\end{align*}
\end{lem}

\begin{lem}\label{lemma:op-r}
For any time warp $f$, $n\in \om{\setminus}\{ 0\}$, and $m\in \om$, 
\begin{align*}
\rr{f}(n) = m &\iff f(m) \leq n < f(m+1)  \\
\rr{f}(n) = \om &\iff f(\om) \leq n \\
\rr{f}(\om) = m &\iff  f(m+1) =  \om \text{ and } \rr{f}(k) = m \text{ for some } k\in \om  \\
\rr{f}(\om) = \om &\iff  f(\om) < \om \text{ or } (f(\om) = \om \text{ and } \forall k\in \om : f(k) < \om).
\end{align*}
\end{lem}

\begin{lem} \label{lemma:op-l}
For any time warp $f$, $n\in \om{\setminus}\{ 0\}$, and $m\in \om$, 
\begin{align*}
\rl{f}(n) = m &\iff f(m-1) < n \leq f(m)  \\
\rl{f}(n) = \om &\iff f(\om) < n \\
\rl{f}(\om) = m &\iff f(m) = \om \text{ and } \rl{f}(k) = m \text{ for some } k\in \om  \\
\rl{f}(\om) = \om &\iff  f(\om) < \om \text{ or } (f(\om) = \om \text{ and } \forall k\in \om : f(k) < \om).
\end{align*}
\end{lem}
 
Note also that $\top = \rl{\bot}$. We call a time warp term {\em basic} if it is constructed using only $\circ$, $\id$, $\bot$, and the defined operations $\rl{t}\defeq \id\rd t$, $\rr{t}\defeq t\ld\id$, and $\ro{t}\defeq \top\ld t$. Our normal form theorem now follows, using \Cref{lemma:reductiontounary} to remove residuals from a time warp term, then~\Cref{lemma:distributivity} and other distributivity properties of $\WarpA$ to push out meets and joins, preserving equivalence in $\WarpA$ at every step.

\begin{thm}\label{theorem:simpleform}
There is an effective procedure that given any time warp term $t$, produces positive integers $m,n_1,\dots,n_m$ and a set of basic time warp terms $\{t_{i,j} \mid 1\le i\le m;\, 1\le j\le n_i\}$ satisfying $\WarpA\models t\eq \bigwedge_{i=1}^m \bigvee_{j=1}^{n_i} t_{i,j}$.
\end{thm}

\begin{cor}\label{corollary:simpleform}
The equational theory of $\WarpA$ is decidable if, and only if, there exists an effective procedure that decides for any finite non-empty set of basic time warp terms $\{t_1,\dots,t_n\}$ if $\WarpA\models \id\le t_1\jn\cdots\jn t_n$.
\end{cor}


We conclude this section by introducing a further notion that will be useful for providing finitary characterizations of time warps.

\begin{dfn}
  For any time warp $f$, let
\[
  \lastw(f)
  \defeq
  \bigwedge
  \{
  p \in \som
  \mid
  f(p) = f(\om)
  \}.
\]
\end{dfn}

\noindent
Observe that $\lastw(f) < \om$ if, and only if,~$f$ is eventually constant, i.e., increases a finite number of times, and that $\lastw(f)$ can be defined equivalently in the language of time warps as $(\rl{f}f)(\omega)$. For future reference, we record the following easy consequences of this definition.

\begin{lem}\label{lemma:last-properties}
For any time warps $f,g$,
\begin{newlist}
\item[\rm (a)]	$\lastw(f g) = \om\iff (\lastw(f) = \om$ and~$\lastw(g) = \om)$
\item[\rm (b)]	$\lastw(f) = \om\iff\lastw(\rr{f}) = \om\iff\lastw(\rl{f}) = \om$.
\end{newlist}
\end{lem}





\section{Diagrams}\label{section:diagrams}

In this section, we define diagrams as finitary characterizations of `potential counterexamples' for equations of the form $\id\le t_1\jn\cdots\jn t_n$, where each $t_i$ is a basic time warp term. This definition is obtained by considering relevant properties of time warps assigned to variables in a refuting valuation, and it therefore follows easily that if $\WarpA\not\models\id \leq t_1 \jn \cdots \jn t_n$, then there exists a suitable refuting diagram. The more challenging direction is to show that every refuting diagram extends to a refuting valuation witnessing $\WarpA\not\models\id \leq t_1 \jn \cdots \jn t_n$.

Note first that, using \Cref{theorem:simpleform}, we may without loss of generality express validity in $\WarpA$ using a simplified language where the restricted residuals are taken as fundamental operations. Let $\TermV$ be a countably infinite set of~\emph{term variables}, with elements denoted by~$x,y,z$, etc.

\begin{dfn}
  A~\emph{basic term} belongs to the grammar
  \[
    \Term \ni
    t, u
    \Coloneqq
    x \mid tu \mid \ro{t} \mid \rl{t} \mid \rr{t}
    \mid \id \mid \bot.
  \]
\end{dfn}

We also define valuations and interpretations explicitly for basic terms.

\begin{dfn}
  A~\emph{valuation}~$\theta$ is a map~$\TermV \to \Warp$. The~\emph{interpretation} of a basic term~$t$ under~$\theta$, denoted by~$\sem{t}_\theta$, is the time warp defined inductively by
  \begin{mathpar}
    \sem{x}_\theta := \theta(x),
    
    \sem{t u}_\theta := \sem{t}_\theta \sem{u}_\theta,

    \sem{t^\star}_\theta := \sem{t}_\theta^\star\, \mbox{ for } \star \in \{ \mathsf{o},\mathsf{\ell}, \mathsf{r} \}.
  \end{mathpar}
\end{dfn}

\noindent
\Cref{corollary:simpleform} tells us that the equational theory of $\WarpA$ is decidable if, and only if, there exists an effective procedure that decides, for any finite set of basic terms $T$, if there exists a valuation $\theta$ and $p\in\som$ such that $\sem{t}_\theta(p) < p$ for all $t\in T$. To refer to this element $p$, we let  $\TimeV$ be a countably infinite set of~\emph{time variables} containing elements denoted by~$\kappa$,~$\kappa'$, etc, noting that in fact only one time variable will be required for the proofs in this paper. We now define a new language of `samples' that will be used to refer to values considered in a diagram.

\begin{dfn}\label{d:sample}
  A~\emph{sample} belongs to the grammar (where $t$ is any basic term)
  \[
    \Time \ni
    \alpha
    \Coloneqq
    \kappa \mid \s{t}{\alpha} \mid \suc(\alpha) \mid \pre(\alpha)
    \mid \last(t).
  \]
\end{dfn}

Although samples are purely syntactic, the notation is indicative of their intended meaning. Given an initial sample set $\{\s{t_1}{\kappa},\ldots,\s{t_n}{\kappa}\}$, obtained from the equation $\id\le t_1 \jn \cdots \jn t_n$, the idea is to `saturate' this set by adding further samples required to describe the existence of a counterexample.

\begin{dfn}\label{d:saturated}
  A sample set~$\Delta$ is called~\emph{saturated} if whenever~$\alpha \in \Delta$ and~$\alpha \leadsto \beta$, also~$\beta \in \Delta$, where~$\leadsto$ is the relation between samples defined by
  \begin{align*}
    \s{t}{\alpha}
    & \leadsto
    \alpha
    &
     \s{\ro{t}}{\alpha}
    & \leadsto
    \s{t}{\alpha}
    \\
    \suc(\alpha)
    & \leadsto
    \alpha
    &
     \s{\rr{t}}{\alpha}
    & \leadsto
    \s{t}{\s{\rr{t}}{\alpha}},
    \s{t}{\suc(\s{\rr{t}}{\alpha})}
    \\
     \pre(\alpha)
    & \leadsto
    \alpha
    &
     \s{\rl{t}}{\alpha}
    & \leadsto
    \s{t}{\s{\rl{t}}{\alpha}},
    \s{t}{\pre(\s{\rl{t}}{\alpha})}
    \\
     \s{tu}{\alpha}
    & \leadsto
    \s{t}{\s{u}{\alpha}}
    &
    \s{t}{\alpha}
    & \leadsto
    \s{t}{\last(t)}.
  \end{align*}
  The~\emph{saturation} of a sample set~$\Delta$ is 
  \[
    \sat{\Delta}
    \defeq
    \{
    \beta
    \mid
    \exists \alpha \in \Delta,
    \alpha \leadsto^* \beta
    \},
  \]
 where  $\leadsto^*$ denotes the  reflexive transitive closure of $\leadsto$.
\end{dfn}

A proof of the following result can be found in~\Cref{section:proof_finite_saturation}.

\begin{lem}\label{lemma:finitesaturation}
The saturation of a finite sample set is finite.
\end{lem}

Let us fix, until after Definition~\ref{definition:diagram}, a saturated sample set $\Delta$.

\begin{dfn}\label{prediagram}
 A~\emph{$\Delta$-prediagram} is a map~$\delta\colon \Delta \to \som$.
\end{dfn}

We now give a list of conditions for a~$\Delta$-prediagram to be a~\emph{$\Delta$-diagram}.

\begin{dfn}
For $p\in \som$, let
\begin{align*}
p\ominus 1 & \defeq \begin{cases}
p-1 & \text{if } p\in \om\mathop{\setminus}\{ 0\} \\
p & \text{if } p\in \{ 0, \om\}
\end{cases},
&
p\oplus 1 & \defeq \begin{cases}
p+1 &\text{if } p \in \om \\
p &\text{if } p= \om
\end{cases}.
\end{align*}
\end{dfn}

\begin{dfn}
  \label{diagram-structural}
   A $\Delta$-prediagram~$\delta$ is called~\emph{structurally-sound} if
  \begin{align}
    \forall \s{t}{\alpha}, \s{t}{\beta} \in \Delta
    &, ~
    \delta(\alpha) \le \delta(\beta) \,\Rightarrow\, \delta(\s{t}{\alpha}) \le
    \delta(\s{t}{\beta})
    \label{c:mon}
    \\
    \forall \s{t}{\alpha} \in \Delta
    &, ~
    \delta(\alpha) = 0 \,\Rightarrow\, \delta(\s{t}{\alpha}) = 0
    \label{c:zero}
    \\
    \forall \pre(\alpha) \in \Delta
    &, ~
    \delta(\pre(\alpha)) = \delta(\alpha) \ominus 1
    \label{c:pre}
    \\
    \forall \suc(\alpha) \in \Delta
    &, ~
    \delta(\suc(\alpha)) = \delta(\alpha) \oplus 1
    \label{c:suc}
    \\
    \forall \s{t}{\alpha} \in \Delta
    &, ~
    \delta(\last(t)) \le \delta(\alpha)
    \,\Leftrightarrow\,
    \delta(\s{t}{\alpha}) = \delta(\s{t}{\last(t)})
    \label{c:last}
    \\
    \forall \s{t}{\last(t)} \in \Delta
    &, ~
    \delta(\last(t))=\om
    \,\Rightarrow\,
    \delta(\s{t}{\last(t)}) = \om
    \label{c:last2}.
  \end{align}
\end{dfn}

\begin{dfn}
  \label{diagram-logical}
  A $\Delta$-prediagram~$\delta$ is called~\emph{logically-sound} if
  \begin{align}
    \forall \s{\id}{\alpha} \in \Delta
    &, ~
    \delta(\s{\id}{\alpha}) = \delta(\alpha)
    \label{c:id}
    \\
    \forall \s{\bot}{\alpha} \in \Delta
    &, ~
    \delta(\last(\bot)) = 0
    \label{c:bot}
    \\
    \forall \s{t u}{\alpha} \in \Delta
    &, ~
    \delta(\s{t u}{\alpha}) = \delta(\s{t}{\s{u}{\alpha}})
    \label{c:prod}
    \\
    \forall \s{t u}{\last(t u)} \in \Delta
    &, ~
    \delta(\last(t u)) = \om
    \,\Rightarrow\,
    \delta(\last(t)) =  \delta(\last(u)) = \om
    \label{c:last-prod1}.
  \end{align}
\end{dfn}

\begin{dfn}
  \label{diagram-o}
   A $\Delta$-prediagram~$\delta$ is called~\emph{$o$-sound} if
  \begin{align}
    \forall \s{\ro{t}}{\alpha} \in \Delta
    &, ~
    \delta(\s{\ro{t}}{\alpha}) = 0 \,\text{ or }\, \delta(\s{\ro{t}}{\alpha}) = \om
    \label{c:o-values}
    \\
    \forall \s{\ro{t}}{\alpha} \in \Delta
    &, ~
    \delta(\alpha) < \om
    \Rightarrow
    (\delta(\s{\ro{t}}{\alpha}) = \om
    \Leftrightarrow
    \delta(\s{t}{\alpha}) = \om)
    \label{c:o-inf}
    \\
    \forall \last(\ro{t}) \in \Delta
    &, ~
    \delta(\last(\ro{t})) < \om
    \label{c:o-last-finite}
    \\
    \forall \s{t}{\alpha}, \s{\ro{t}}{\last(\ro{t})} \in \Delta
    &, ~
    (\delta(\s{\ro{t}}{\last(\ro{t})}) < \om \,\text{ and }\, \delta(\alpha) < \om)
    \,\Rightarrow\,
    \delta(\s{t}{\alpha}) < \om.
    \label{c:o-last-value-finite}
  \end{align}
\end{dfn}

\begin{dfn}
  \label{diagram-r}
   A $\Delta$-prediagram~$\delta$ is called~\emph{$r$-sound} if
  \begin{align}
    \forall \s{t}{\s{\rr{t}}{\alpha}} \in \Delta
    &, ~
    \delta(\s{t}{\s{\rr{t}}{\alpha}}) \le \delta(\alpha)
    \label{c:r-lower}
    \\
    \forall \s{\rr{t}}{\alpha}\in \Delta
    &,~
    (0<\delta(\alpha) < \om\,\text{ and }\,\delta(\s{\rr{t}}{\alpha}) < \om)
     \,\Rightarrow\, \delta(\alpha) < \delta(\s{t}{\suc(\s{\rr{t}}{\alpha}) } ) 
    \label{c:r-finite}
    \\
    \forall\s{\rr{t}}{ \last(\rr{t})} \in \Delta
    &, ~
    \delta(\last(\rr{t})) = \om
    \,\Rightarrow\,
    \delta(\last(t)) = \om
    \label{c:r-last-inf}
    \\
    \forall \s{\rr{t}}{\last(\rr{t})} \in \Delta
    &, ~
    \delta(\s{\rr{t}}{ \last(\rr{t})}) < \om
    \,\Rightarrow\,
    \delta(\s{t}{\suc(\s{\rr{t}}{ \last(\rr{t})})}) = \om.
    \label{c:r-last-value-finite}
  \end{align}
\end{dfn}

\begin{dfn}
  \label{diagram-l}
   A $\Delta$-prediagram~$\delta$ is called~\emph{$\ell$-sound} if
  \begin{align}
    \forall \s{t}{\s{\rl{t}}{\alpha}} \in \Delta
    &, ~
    \delta(\s{\rl{t}}{\alpha}) < \om
    \,\Rightarrow\,
    \delta(\alpha)  \le \delta(\s{t}{\s{\rl{t}}{\alpha}})
    \label{c:l-finite1}
    \\
    \forall \s{\rl{t}}{\alpha}\in \Delta
    &, ~
    (0 < \delta(\alpha) < \om \,\text{ and }\, \delta(\s{\rl{t}}{\alpha}) < \om) 
    \,\Rightarrow\, \delta(\s{t}{{\pre(\s{\rl{t}}{\alpha})}})<  \delta(\alpha)
    \label{c:l-finite2}
    \\
    \forall \s{t}{\s{\rl{t}}{\alpha}} \in \Delta
    &, ~
    (\delta(\alpha) < \om \,\text{ and }\, \delta(\s{\rl{t}}{\alpha}) = \om)
    \,\Rightarrow\,
    \delta(\s{t}{\s{\rl{t}}{\alpha}})  < \delta(\alpha)
    \label{c:l-inf}
    \\
    \forall \s{\rl{t}}{ \last(\rl{t})} \in \Delta
    &, ~
    \delta(\last(\rl{t})) = \om
    \,\Rightarrow\,
    \delta(\last(t)) = \om
    \label{c:l-last-inf}
    \\
    \forall \s{\rl{t}}{\last(\rl{t})}\in \Delta
    &, ~
    \delta(\s{\rl{t}}{\last(\rl{t})}) < \om
    \,\Rightarrow\,
    \delta(\s{t}{ \s{\rl{t}}{\last(\rl{t})}}) = \om.
    \label{c:l-last-value-finite}
  \end{align}
\end{dfn}

\begin{dfn}
  \label{definition:diagram}
  A~$\Delta$-prediagram~$\delta$ is called a \emph{$\Delta$-diagram} if it is structurally sound, logically sound,~$o$-sound, $\ell$-sound, and~$r$-sound.
\end{dfn}

It follows from the next proposition that any counterexample to the validity of an equation in $\WarpA$ restricts to a finite diagram witnessing this failure. More precisely, if $\WarpA\not\models\id \leq t_1 \jn \cdots \jn t_n$, where each $t_i$ is a basic term, and $\Delta$ is the saturation of the sample set $\{\s{t_1}{\kappa},\ldots,\s{t_n}{\kappa}\}$, then there exists a~$\Delta$-diagram~$\delta$ satisfying $\delta(\kappa)>\delta(t_i[\kappa])$ for each $i\in\{1,\dots,n\}$.

\begin{prop}\label{proposition:valuation-to-diagram}
Let~$T$ be a set of basic terms, $\kappa$ a time variable, and $\Delta$ the saturation of the sample set $\{\s{t}{\kappa}\mid t\in T\}$. Then for any valuation $\theta$ and $p\in\som$, there exists a~$\Delta$-diagram~$\delta$ such that~$\delta(\kappa)=p$ and~$\delta(\s{t}{\kappa})=\sem{t}_\theta(p)$ for all $t\in T$.
\end{prop}
\begin{proof}
 We define the map $\delta\colon \Delta \to \som$ recursively by
  \[
    \begin{array}{rrl}
      & \delta(\kappa)
      & \defeq p
      \\
      \forall \s{t}{\alpha} \in \Delta,
      &
      \delta( \s{t}{\alpha}) & \defeq  \sem{t}_\theta(\delta(\alpha))
      \\
      \forall \last(t) \in \Delta,
      &
      \delta(\last(t)) & \defeq \lastw(\sem{t}_\theta)
      \\
      \forall \pre(\alpha) \in \Delta,
      &
      \delta(\pre(\alpha)) & \defeq \delta(\alpha) \ominus 1
      \\
      \forall \suc(\alpha) \in \Delta,
      &
      \delta(\suc(\alpha)) & \defeq  \delta(\alpha) \oplus 1.
    \end{array}
  \]
The map $\delta$ is well-defined since $\alpha \in \Delta$ if, and only if, there exist samples $\alpha_1,\ldots,\alpha_n$ such that $\alpha_1 = \s{t}{\kappa}$ for some $t\in T$, $\alpha_n = \alpha$, and $\alpha_j \leadsto \alpha_{j+1}$ for each $j\in\{1,\ldots,n-1\}$. So $\delta$ is a $\Delta$-prediagram. A proof that $\delta$ is a $\Delta$-diagram---i.e., that $\delta$ satisfies conditions~\Crefrange{c:mon}{c:l-last-value-finite}---may be found in~\Cref{section:proof_delta_is_a_diagram}.
\end{proof}

We now turn our attention to proving that every $\Delta$-diagram $\delta$ extends to a valuation $\theta$ satisfying $\sem{t}_\theta(\delta(\alpha)) = \delta(\s{t}{\alpha})$ for all $t[\alpha] \in \Delta$. First, we use $\delta$ to define a partial sup-preserving function $\diag{t}{\delta}$ for each basic term $t$.

\begin{dfn}
  \label{definition:points}
  For any $\Delta$-diagram $\delta$ and basic term~$t$, let
  \begin{mathpar}
    \diag{t}{\delta}
    \defeq
    \{
    (\delta(\alpha), \delta(\s{t}{\alpha}))
    \mid
    \s{t}{\alpha} \in \Delta
    \}.
  \end{mathpar}
A time warp~$f$ \emph{extends} $\diag{t}{\delta}$  if~$f(i) = j$ for all~$(i, j) \in \diag{t}{\delta}$, and \emph{strongly extends} $\diag{t}{\delta}$ if also  			\[
			\text{either }\diag{t}{\delta} =\emptyset\enspace\text{or}\enspace
			(\diag{t}{\delta}\neq\emptyset \,\text{ and }\, \delta(\last(t)) = \om \implies \lastw(f) = \om).
			\]
\end{dfn}

\begin{lem}\label{lemma:nonempty}
 There exists an effective procedure that produces for any finite $\Delta$-diagram $\delta$ and term variable $x$, an algorithmic description of a  time warp~$f$ that strongly extends~$\diag{x}{\delta}$.
\end{lem}

\begin{proof}
If~$\diag{x}{\delta}=\emptyset$, then any time warp strongly extends it, so assume $\diag{x}{\delta}\neq\emptyset$.  By~\cref{c:mon},~$\diag{x}{\delta}$ can be considered as a partial map from~$\som$ to $\som$. Moreover, since~$\Delta$ is saturated, and, by~\cref{c:last},~$\delta(\s{x}{\last(x)}) \ge j$ for all~$(i, j) \in \diag{x}{\delta}$, we have~$(\delta(\last(x)), \delta(\s{x}{\last(x)})) \in \diag{x}{\delta}$.

Let~$X \coloneqq \diag{x}{\delta} \cup \{ (0, 0), (\om,\delta(\s{x}{\last(x)})) \}$. This is still a partial map by~\cref{c:zero} and~\cref{c:last}. For each $i \in\om$, there exists a unique pair~$(i_1, j_1), (i_2, j_2) \in X$ such that~$i_1 \le i < i_2$ and there is no~$(i_3, j_3) \in X$ with~$i_1 < i_3 < i_2$, and we define
 \[
 f(i) \defeq \min(j_2, j_1 \oplus (i - i_1)),
 \]
 where $n\oplus m := \min\{\omega, n+m \}$. Let also $f(\om) \defeq \delta(\s{x}{\last(x)}))$.

Clearly~$f$ is monotonic. It extends~$\diag{x}{\delta}$, since~$i = i_1 < \om$ implies~$f(i_1) = \min(j_2, j_1) = j_1$. In particular, $f(0) = 0$. To confirm that $f$ is a time warp, it remains to show  that~$f(\om) = \bigvee\{f(i)\mid i\in\om\}$. If~$\delta(\s{x}{\last(x)}) =f(\om) < \om$, then, by \cref{c:last2}, $\delta(\last(x)) < \om$ and, by monotonicity, $f(i) = f(\om)$ for each~$i \ge \delta(\last(x))$ and~$f(\om) = f(\delta(\last(x))) = \bigvee\{f(i)\mid i\in\om\}$.  If~$f(\om) = \om$, then for each~$j\in\om$, there exists an~$i\in\om$ such that~$f(i) > j$, and hence~$\bigvee_{i < \om} f(i) = \om = f(\om)$. 

Finally, suppose that~$\delta(\last(x)) = \om$. Then \cref{c:last2} yields $(\om,\om) \in \diag{x}{\delta}$ and for any $(i,j) \in \diag{x}{\delta}$, if $i\in\om$, then also $j\in\om$. Hence, $\lastw(f) = \om$, by the definition of $f$. So $f$ strongly extends~$\diag{x}{\delta}$.
\end{proof}

\begin{lem} \label{lemma:fundamental}
For every basic term~$t$, valuation~$\theta$, and~$\Delta$-diagram~$\delta$, if~$\theta(x)$ strongly extends~$\diag{x}{\delta}$ for every term variable~$x$, then~$\sem{t}_\theta$ strongly extends~$\diag{t}{\delta}$.
\end{lem}
\begin{proof}
 By induction on~$t$. The case~$t = x$ is immediate and the other cases follow by a series of lemmas proved in~\Cref{section:fundamental}, and the induction hypothesis. 
\end{proof}

The next proposition is then a direct consequence of~Lemmas~\ref{lemma:nonempty} and~\ref{lemma:fundamental}.  

\begin{prop}\label{proposition:diagram-to-valuation}
There is an effective procedure that produces for any finite $\Delta$-diagram $\delta$, an algorithmic description of a valuation $\theta$ satisfying $\sem{t}_\theta(\delta(\alpha)) = \delta(\s{t}{\alpha})$ for all $t[\alpha] \in \Delta$.
\end{prop}

We are now ready to establish the main theorem of this section.

\begin{thm}
  \label{theorem:main-diagram-theorem}
Let $t_1,\ldots,t_n$ be basic terms, $\kappa$ a time variable, and $\Delta$ the saturation of the sample set  $\{\s{t_1}{\kappa},\ldots,\s{t_n}{\kappa}\}$. Then $\WarpA\not\models\id \leq t_1 \jn \cdots \jn t_n$ if, and only if, there exists a $\Delta$-diagram $\delta$ such that $\delta(\kappa) > \delta(\s{t_i}{\kappa})$ for all $i\in \{1,\ldots,n \}$. 
\end{thm}
\begin{proof}
Suppose first that $\WarpA\not\models\id \leq t_1 \jn \cdots \jn t_n$. Then there exist a valuation $\theta$ and $p\in \som$ such that $p = \id(p) > \sem{t_i}_\theta(p)$ for all  $i\in \{1,\ldots,n \}$. Hence, by \Cref{proposition:valuation-to-diagram}, there exists a $\Delta$-diagram $\delta$ such that  $\delta(\kappa) = p > \sem{t_i}_\theta(p) = \delta(\s{t_i}{\kappa})$ for all  $i\in \{1,\ldots,n \}$. 

Now suppose that there exists a $\Delta$-diagram $\delta$ such that $\delta(\kappa) > \delta(\s{t_i}{\kappa})$ for all $i\in \{1,\ldots,n \}$.  Then, by \Cref{proposition:diagram-to-valuation}, there exists a valuation $\theta$ such that $\sem{t_i}_\theta(\delta(\kappa)) = \delta(\s{t_i}{\kappa})$ for all $i\in \{ 1,\ldots,n\}$. So $\id(\delta(\kappa) ) = \delta(\kappa) > \sem{t_i}_\theta(\delta(\kappa))$ for all $i\in \{ 1,\ldots,n\}$. Hence  $\WarpA\not\models\id \leq t_1 \jn \cdots \jn t_n$. 
\end{proof}





\section{Decidability via Logic}\label{section:logic}

Let $t_1, \dots, t_n$ be basic terms, $\kappa$ a time variable, and $\Delta$ the saturation of the sample set $\{t_1[\kappa], \dots, t_n[\kappa]\}$. Our aim in this section is to express the existence of a $\Delta$-diagram witnessing $\WarpA\not\models\id \leq t_1 \jn \ldots \jn t_n$, as stated in \Cref{theorem:main-diagram-theorem}, via an existential sentence over the natural numbers with the ordering and successor relations. Since the first-order theory of this structure is decidable, it follows that the equational theory of $\WarpA$ is decidable, concluding the proof of \Cref{theorem:main_result}.

Note that in the logic encoding, we will no longer allow $\om$ as a value for the variables. The theoretical reason why this is possible is that the ordinal $\som$ admits a first-order (even quantifier-free) interpretation in $\om$. However, we will avoid relying upon such model-theoretic generalities here and just give the necessary concrete definitions.

Our construction of a first-order formula $\phi$ encoding the existence of a $\Delta$-diagram uses the samples in $\Delta$ as variables and proceeds in two steps:

\begin{newlist}
\item[1.] We define a formula  $\psi$ with variables in $\Delta$, intended to be interpreted in $\som$, using the order relation symbol $\aleq$, the successor relation symbol $\aS$, and two further unary relation symbols $\aO$ and $\I$, where the intended interpretations of $\aO(x)$ and $\I(x)$  are ``$x = \om$'' and  ``$x = 0$'', respectively.
\item[2.] We obtain $\phi$ by eliminating the symbols $\aO$ and $\I$ from $\psi$ and re-interpreting  $\aleq$ and $\aS$ using an encoding of $\som$ in the structure $(\mathbb{N}, \leq, S, 0)$.
\end{newlist}

Let $\tau$ be the relational first-order signature with two binary relation symbols $\preceq$ and $\aS$, and two unary relation symbols $\aO$ and $\I$. We consider $\som$ as a $\tau$-structure by defining $\aleq^{\som}$ to be the natural ordering of $\som$, $\aS^{\som} \defeq \{(n, n+1)\mid n\in\om \} \cup \{(\om,\om)\}$, $\I^{\som} \defeq \{0\}$, and $\aO^{\som} \defeq \{\om\}$. Note that a  $\Delta$-prediagram is a valuation of the variables in $\Delta$ in this structure. 

We define $\psi$ by translating the defining properties of being a $\Delta$-diagram into quantifier-free formulas of first-order logic in the signature $\tau$ with variables from $\Delta$. In the following definition, the symbols $\andd$ and $\orr$ denote the logical connectives `and' and `or', respectively, and the notation $a \ale b$ is shorthand for $a \aleq b \andd \neg(b \aleq a)$. Note also that $\psi$ is well-defined, since $\Delta$ is finite by \Cref{lemma:finitesaturation}.

\begin{dfn}\label{definition:psi}
Let $\psi$ be the first-order quantifier-free $\tau$-formula 
\[
\mathlarger{\andd} (\mathsf{struct} \cup \mathsf{log} \cup \mathsf{bounds} \cup \mathsf{right} \cup \mathsf{left} \cup \mathsf{fail}),
\]
where the first five sets, corresponding to Definitions~\ref{diagram-structural}-\ref{diagram-l} in the definition of a diagram, and $\mathsf{fail}$,  expressing the failure of  $\id \leq t_1 \jn \dots \jn t_n$ in $\WarpA$ at the time variable $\kappa$, are defined as follows:
\begin{align*}
  \mathsf{struct}  \defeq \; &\{ \var{\alpha} \aleq \var{\beta} \To \var{t[\alpha]} \aleq \var{t[\beta]}\mid t[\alpha], t[\beta] \in \Delta\} \; \cup \\
                  &\{ \I(\var{\alpha}) \To \I(\var{t[\alpha]})\mid t[\alpha] \in \Delta\} \; \cup \\
                   &\{ \aS(\var{\pre(\alpha)}, \var{\alpha}) \orr (\I(\var{\pre(\alpha)}) \andd \I(\var{\alpha}))\mid \pre(\alpha) \in \Delta\} \; \cup \\
                   &\{ \aS(\var{\alpha}, \var{\suc(\alpha)})\mid \suc(\alpha) \in \Delta\} \; \cup \\
                   &\{ \var{\last(t)} \aleq \var{\alpha} \Leftrightarrow \var{t[\alpha]} \aeq \var{t[\last[t]]} \mid t[\alpha] \in \Delta \} \; \cup \\
                   &\{ \aO(\var{\last(t)}) \To \aO(\var{t[\last(t)]})\mid t[\last(t)] \in \Delta \}\\[.1in]
\mathsf{log} \defeq \; &\{ \var{\id[\alpha]} \aeq \var{\alpha}\mid \id[\alpha] \in \Delta \} \; \cup \\
                    &\{ \I(\var{\last(\bot)})\mid \s{\bot}{\alpha} \in \Delta \} \; \cup \\
                    &\{\var{tu[\alpha]} \aeq \var{t[u[\alpha]]}\mid tu[\alpha] \in \Delta \} \; \cup \\
                    &\{ \aO(\var{\last(tu)}) \To (\aO(\var{\last(t)}) \andd \aO(\var{\last(u)}))\mid tu[\last(tu)] \in \Delta \}\\[.1in]
\mathsf{bounds} \defeq \; &\{ \I(\var{ \s{\ro{t}}{\alpha} }) \orr \aO(\var{ \s{\ro{t}}{\alpha} })\mid \s{\ro{t}}{\alpha} \in \Delta\} \; \cup \\ 
  &\{ \neg \aO(\var{\alpha}) \Rightarrow ( \aO( \var{\s{\ro{t}}{\alpha}} ) \Leftrightarrow \aO ( \var{\s{t}{\alpha}} ) )\mid  \s{\ro{t}}{\alpha} \in \Delta \} \; \cup \\  %
&\{ \neg \aO(\var{\last(\ro{t})})\mid \last(\ro{t}) \in \Delta \} \; \cup \\ 
&\{ ( \neg \aO ( \var{\s{\ro{t}}{\last(\ro{t})}} ) \andd \neg \aO ( \var{\alpha} ) ) \Rightarrow \neg \aO ( \var{\s{t}{\alpha}} )\mid \s{t}{\alpha}, \s{\ro{t}}{\last(\ro{t})} \in \Delta \}\\[.1in] 
\mathsf{right} \defeq \; &\{ \var{ \s{t}{\s{\rr{t}}{\alpha}} } \aleq \var{\alpha}\mid  \s{t}{\s{\rr{t}}{\alpha}} \in \Delta \} \cup \\ 
            &\{ ( \neg \I(\var{\alpha}) \andd \neg \aO(\var{\alpha}) \andd \neg \aO ( \var{ \s{\rr{t}}{\alpha} } ) \Rightarrow \var{\alpha} \ale \var{ \s{t}{ \suc( \s{\rr{t}}{\alpha}) } }\mid \s{t}{\suc(\s{\rr{t}}{\alpha} )}\in \Delta \} \; \cup \\ 
  & \{ \aO(\var{\last(\rr{t})}) \Rightarrow \aO(\var{\last(t)})\mid 
    \s{\rr{t}}{\last(\rr{t})} \in \Delta \} \; \cup \\ 
  &\{ \neg\aO(\var{ \s{\rr{t}}{ \last(\rr{t})} } ) \Rightarrow \aO ( \var{ \s{t}{\suc(\s{\rr{t}}{ \last(\rr{t})})}} )\mid \s{t}{\suc(\s{\rr{t}}{\last(\rr{t})})} \in \Delta \}\\[.1in] 
\mathsf{left} \defeq \; &\{ \neg \aO( \var{ \s{\rl{t}}{\alpha} } ) \Rightarrow \var{\alpha} \aleq \var{\s{t}{\s{\rl{t}}{\alpha}} }\mid \s{t}{\s{\rl{t}}{\alpha}} \in \Delta\} \; \cup \\            &\{ ( \neg \I ( \var{ \alpha } ) \andd \neg \aO( \var{\alpha } ) \andd \neg \aO ( \var{ \s{\rl{t}}{\alpha} } ) )
              \Rightarrow
              \var {\s{t}{{\pre(\s{\rl{t}}{\alpha})}}} \ale \var{\alpha} 
             \mid \s{t}{\pre(\s{\rl{t}}{\alpha} )}\in \Delta
              \} \; \cup \\ 
            &\{
              (\neg\aO(\var{\alpha}) \andd \aO (\var{\s{\rl{t}}{\alpha}} ) )
              \Rightarrow
              \var{\s{t}{\s{\rl{t}}{\alpha}}}  \ale \var{\alpha}
             \mid \s{t}{\s{\rl{t}}{\alpha}} \in \Delta
              \} \; \cup \\ 
            &\{
              \aO(\var{\last(\rl{t})})
              \Rightarrow
              \aO(\var{\last(t)})
             \mid \s{\rl{t}}{\last(\rl{t})} \in \Delta
              \} \; \cup \\ 
  &\{ \neg\aO(\var\{\s{\rl{t}}{\last(\rl{t})}\}
     \Rightarrow
    \aO(\var{\s{t}{ \s{\rl{t}}{\last(\rl{t})}}})
    \mid \s{\rl{t}}{\last(\rl{t})}\in \Delta \} \\[.1in]
\mathsf{fail} \defeq \; & \{ \var{\s{t_i}{\kappa}} \ale \var{\kappa}\mid 1 \leq i \leq n \}.
\end{align*}
\end{dfn}

The next lemma then follows directly from the definition of a $\Delta$-diagram.

\begin{lem}\label{lemma:omega-plus-encoding}
Let $\delta \colon \Delta \to \som$ be a $\Delta$-prediagram. Then $\som, \delta \models \psi$ if, and only if, $\delta$ is a $\Delta$-diagram such that $\delta(\s{t_i}{\kappa}) < \delta(\kappa)$ for each $i \in \{1,\dots,n\}$.
\end{lem}

\Cref{theorem:main-diagram-theorem} and \Cref{lemma:omega-plus-encoding} together show that $\WarpA\not\models\id \leq t_1 \jn \ldots \jn t_n$ if, and only, if $\psi$ is satisfiable in $\som$. We could therefore conclude the proof of \Cref{theorem:main_result} at this point by appealing to classical decidability results on the first-order theory of ordinals \cite{LL66}. Instead, however, we show explicitly how to interpret the $\tau$-structure $\som$ inside the standard model $(\mathbb{N}, \leq, S, 0)$, which is more commonly available in satisfiability solvers.

Consider the first-order signature $\sigma$ with two binary relation symbols $\leq$ and $S$, and one constant symbol $0$, and let $\mathbb{N}$ denote the $\sigma$-structure based on the natural numbers, where $\leq^{\mathbb{N}}$ is the usual order, $S^{\mathbb{N}} \defeq \{(n, n + 1) \mid n \in \mathbb{N} \}$, and $0^{\mathbb{N}} \defeq 0$. The following definition and  lemma contain the crucial observation needed for encoding $\tau$-formulas over $\som$ into $\sigma$-formulas over $\mathbb{N}$.\footnote{We thank Thomas Colcombet for suggesting this idea.}

\begin{dfn}
Define the bijection $\iota \colon \mathbb{N} \to \som$ by $\iota(0) \defeq \om$, and $\iota(n) \defeq n - 1$ for each $n\in\om\mathop{\setminus}\{0\}$.
\end{dfn}
For any valuation $w \colon \Delta \to \mathbb{N}$, let $\hat{w} \colon \Delta \to \som$ denote the function defined by $\hat{w}(x) \defeq \iota(w(x))$. Note that the map $w \mapsto \hat{w}$ is a bijection between $\mathbb{N}^\Delta$ and $(\som)^\Delta$, since  $\iota$ is a bijection.

\begin{lem}\label{lemma:encode-in-omega}
  Let $\chi$ be a quantifier-free $\tau$-formula. Define $\chi'$ to be the quantifier-free $\sigma$-formula obtained from $\chi$ by making the following symbolic substitutions for every occurrence of an atomic formula in $\chi$:
  \begin{enumerate}
  \item[\rm (i)] $\aO(x)$ is replaced by $x = 0$
  \item[\rm (ii)] $\I(x)$ is replaced by $S(0,x)$
  \item[\rm (iii)] $\aS(x,y)$ is replaced by $(x = 0 \andd y = 0) \orr (\neg(x = 0) \andd S(x,y))$
  \item[\rm (iv)] $x \aleq y$ is replaced by $y = 0 \orr (\neg(x = 0) \andd x \leq y)$.
  \end{enumerate}
  Then, for any valuation $w \colon \Delta \to \mathbb{N}$,  $\mathbb{N}, w \models \chi'$ if, and only if, $\som, \hat{w} \models \chi$.
\end{lem}
\begin{proof}
  By induction on the complexity of $\chi$. The induction step is immediate, and the atomic cases essentially follow from the definitions; we just show the proof for $x \aleq y$ as an example. For any valuation $w$, we have $\som, \hat{w} \models x \aleq y$ if, and only if, $\hat{w}(y) = \om$ or ($\hat{w}(x) \neq \omega$ and $\hat{w}(x) \leq \hat{w}(y)$) in $\som$. Using the definition of $\hat{w}$, this is equivalent to $w(y) = 0$ or ($w(x) \neq 0$ and $w(x) \leq w(y)$) in $\mathbb{N}$, that is, $\mathbb{N}, w \models y = 0 \orr (\neg(x = 0) \andd x \leq y)$.
\end{proof}

Finally, we define our quantifier-free $\sigma$-formula $\phi$ encoding the non-validity of $\id \leq t_1 \jn \cdots \jn t_n$ in $\WarpA$.

\begin{dfn}
Let $\phi \defeq \psi'$, the $\sigma$-formula obtained from the $\tau$-formula $\psi$ (\Cref{definition:psi}) by performing the replacements in \Cref{lemma:encode-in-omega}.
\end{dfn}

We are now ready to put everything together.

\begin{thm}\label{theorem:main-logic}
  The time warp equation $\id \leq t_1 \jn \cdots \jn t_n$ is valid in $\WarpA$ if, and only if, the quantifier-free $\sigma$-formula $\phi$ is unsatisfiable in $\mathbb{N}$. Moreover, any valuation $w \colon \Delta \to \mathbb{N}$ such that $\mathbb{N}, w \models \phi$ effectively yields a valuation $\theta$ of the time warp variables occurring in $t_1 \jn \cdots \jn t_n$ such that $\WarpA, \theta \models \id \nleq t_1 \jn \cdots \jn t_n$.
\end{thm}
\begin{proof}
  By \Cref{theorem:main-diagram-theorem}, the equation $\id \leq t_1, \jn \cdots \jn t_n$ is not valid in $\WarpA$ if, and only if, there exists a $\Delta$-diagram $\delta$ such that $\delta(\s{t_i}{\kappa}) < \delta(\kappa)$ for all $i\in\{1,\dots,n\}$. By \Cref{lemma:omega-plus-encoding}, the latter is equivalent to the existence of a valuation $v \colon \Delta \to \som$ such that $\som, v \models \psi$. By \Cref{lemma:encode-in-omega}, the latter is in turn equivalent to the existence of a valuation $w \colon \Delta \to \mathbb{N}$ such that $\mathbb{N}, w \models \phi$.

  For the second claim, we retrace our steps. If $w \colon \Delta \to \mathbb{N}$ is a valuation such that $\mathbb{N}, w \models \phi$, define the function $\delta \colon \Delta \to \som$ by $\delta(\alpha) \defeq \iota(w(\var{\alpha}))$ for $\alpha \in \Delta$. By \Cref{lemma:omega-plus-encoding}, $\delta$ is a $\Delta$-diagram such that $\delta(\s{t_i}{\kappa}) < \delta(\kappa)$ for each $i\in\{1,\dots,n\}$. By \Cref{proposition:diagram-to-valuation}, $\delta$ effectively yields a valuation $\theta$ that falsfies $\id \leq t_1 \jn \cdots \jn t_n$.
\end{proof}

\Cref{theorem:main_result} follows now directly from \Cref{theorem:main-logic} and the decidability of the first-order theory of $\mathbb{N}$ (see, e.g., \cite{LL66}).


\subsubsection*{Concluding remark.}

The proof of~\Cref{theorem:main-logic}, together with the normal form results of~\Cref{section:normalform}, provides a decision procedure for the equational theory of the time warp algebra, as explained in~\Cref{section:introduction}. We are currently in the process of implementing this decision procedure in a software tool. This tool is written in the OCaml functional programming language~\cite{OCaml412} and uses the Z3 theorem prover~\cite{Z3} to decide the satisfiability of the final logic formula. Our experiments with a preliminary implementation for basic time warp terms have been encouraging so far, and we hope to integrate a full version in a compiler for graded modalities. From a complexity perspective, the most challenging issue here is to deal with the potentially very large saturated sample sets and corresponding logic formulas produced by time warp equations. We therefore intend to consider encodings of the decision problem for time warps using alternative, possibly more efficient, data structures such as---following a helpful suggestion of one of the referees of this paper---{\em arrays}  (see~\cite{BMS06}) that are also supported by the Z3 theorem prover.




\bibliographystyle{splncs04}
\bibliography{bibliography_combined}

\newpage



\appendix
\section{Appendix}


\subsection{Proof of~\Cref{lemma:finitesaturation}}\label{section:proof_finite_saturation}

\begin{dfn}
The sample $\gamma_\alpha^\beta$ is defined inductively for samples $\alpha, \beta, \gamma$ by
\begin{align*}
\kappa_\alpha^\beta & \defeq \begin{cases}
\beta &\text{if } \kappa = \alpha \\
\kappa & \text{otherwise};
\end{cases} 
&\last(t)_\alpha^\beta & \defeq \begin{cases}
\beta &\text{if } \last(t) = \alpha \\
\last(t) & \text{otherwise};
\end{cases}  \\
 \s{t}{\gamma}_\alpha^\beta & \defeq \begin{cases}
\beta &\text{if } \s{t}{\gamma} = \alpha \\
\s{t}{\gamma_\alpha^\beta}& \text{otherwise};
\end{cases}  
&{\sf q}(\gamma)_\alpha^\beta & \defeq \begin{cases}
\beta &\text{if }  {\sf q}(\gamma) = \alpha \\
 {\sf q}(\gamma_\alpha^\beta)& \text{otherwise}
\end{cases} \quad \text{for } {\sf q} \in \{ \suc,\pre \}. 
\end{align*}
Note that $(\gamma_\alpha^\beta)_\beta^\alpha = \gamma$.
\end{dfn}

\begin{dfn}
For samples $\alpha,\beta_1,\ldots,\beta_k$, let $\mu(\beta_1,\ldots,\beta_k)   \defeq \lvert \{ \beta_1,\ldots,\beta_k \}^\leadsto \rvert$ and $\mu_\alpha(\beta_1,\ldots,\beta_k) \defeq\lvert M_\alpha(\beta_1,\ldots,\beta_k) \rvert$, where $M_\alpha(\beta_1,\ldots,\beta_k)$ denotes the set of samples $\beta \in  \{ \beta_1,\ldots,\beta_k\}^\leadsto$ such that whenever $\alpha_1 \leadsto \cdots \leadsto \alpha_n  $ with $\alpha_1 = \beta_j$ and $\alpha_n = \beta$, there exists an $i\in \{1,\ldots,n \}$ such that $\alpha_i = \alpha$.
\end{dfn}

Note that clearly $\mu(\alpha_1,\ldots,\alpha_k) \leq \mu(\alpha_1)+\ldots + \mu(\alpha_k)$.

\begin{lem}\label{lemma:saturation-inequality}
For any basic term $t$, ${\sf q} \in \{ \suc,\pre \}$, samples $\alpha,\gamma_1,\gamma_2$, and time variable $\kappa$,
\begin{align*}
\mu(\s{t}{\alpha},\s{t}{{\sf q}(\alpha)},\gamma_1,\gamma_2)& \leq \mu(\s{t}{\kappa},\s{t}{{\sf q}(\kappa)},\gamma_1,\gamma_2) + \mu_\alpha(\s{t}{\alpha},\s{t}{{\sf q}(\alpha)},\gamma_1,\gamma_2)\\
\mu(\s{t}{\alpha}) & \leq \mu(\s{t}{\kappa}) + \mu(\alpha).
\end{align*} 
In particular, $\mu(\s{t}{\last(u)}) \leq \mu(\s{t}{\kappa})$ for any basic term $u$.
\end{lem}

\begin{proof}
If $\alpha = \last(u)$ for some basic term $u$, then clearly even the inequality $\mu(\s{t}{\alpha},\s{t}{{\sf q}(\alpha)},\gamma_1,\gamma_2) \leq \mu(\s{t}{\kappa},\s{t}{{\sf q}(\kappa)},\gamma_1,\gamma_2)$ holds. Suppose that $\alpha \neq \last(u)$. Let  $A \defeq \{\s{t}{\kappa},\s{t}{{\sf q}(\kappa)}\} \cup \{\gamma_1,\gamma_2\}^\leadsto \cup M_\alpha(\s{t}{\kappa},\s{t}{{\sf q}(\kappa)},\gamma_1,\gamma_2)$, where we assume for convenience of notation that these unions are disjoint. Define the function $K$ from $A$ to the set of all samples by 
\[
K(\beta) 
\defeq \begin{cases}
\beta &\text{if } \beta \in M_\alpha(\s{t}{\kappa},\s{t}{{\sf q}(\kappa)},\gamma_1,\gamma_2) \cup \{\gamma_1,\gamma_2\}^\leadsto \\
\beta_\kappa^\alpha &\text{if } \beta \in\{\s{t}{\kappa},\s{t}{{\sf q}(\kappa)}\}^\leadsto.
\end{cases}
\]
It suffices to show that $  \{\s{t}{\alpha},\s{t}{{\sf q}(\alpha)},\gamma_1,\gamma_2\}^\leadsto$ is contained in the image of $K$.
Let $\beta \in \{\s{t}{\alpha},\s{t}{{\sf q}(\alpha)},\gamma_1,\gamma_2\}^\leadsto$. 
If $\beta \in \{\gamma_1,\gamma_2\}^\leadsto$, then clearly $\beta$ is in the image of $K$. So we may assume that $\beta \in\{\s{t}{\alpha},\s{t}{{\sf q}(\alpha)}\}^\leadsto \setminus \{\gamma_1,\gamma_2\}^\leadsto $.  
Then either there exist $\alpha_1,\ldots,\alpha_n$ with $\alpha_1 \in \{\s{t}{\alpha},\s{t}{{\sf q}(\alpha)}\}$, $\alpha_n = \beta$, and $\alpha_1 \leadsto \ldots \leadsto \alpha_n$ such that  $\alpha_i \neq \alpha$ for all $i \in \{ 1,\ldots,n \}$, or not.
If not, then $\beta \in M_\alpha(\s{t}{\kappa},\s{t}{{\sf q}(\kappa)},\gamma_1,\gamma_2)$, i.e.,  $\beta = K(\beta)$ is in the image of $K$. 
Otherwise we want to show that  ${\alpha_1}_\alpha^\kappa \leadsto \ldots \leadsto {\alpha_n}_\alpha^\kappa$. 
Then, since  ${\alpha_1}_\alpha^\kappa =\{\s{t}{\kappa},\s{t}{{\sf q}(\kappa)}\}$,  we have $\beta^\kappa_\alpha \in \{\s{t}{\kappa},\s{t}{{\sf q}(\kappa)}\}^\leadsto $ and $\beta = K({\beta}_\alpha^\kappa)$ is in the image of $K$.
 We prove the claim by induction on $n$. If $n= 1$, then there is nothing to prove. 
 Suppose that the claim is proved for $n$ and we have  $\alpha_1  \in  \{\s{t}{\alpha},\s{t}{{\sf q}(\alpha)}\}$, $ \alpha_{n+1} = \beta$, and $\alpha_1 \leadsto \ldots \leadsto \alpha_n \leadsto \alpha_{n+1}$. 
 By the induction hypothesis we get  ${\alpha_1}_\alpha^\kappa \leadsto \ldots \leadsto {\alpha_n}_\alpha^\kappa$.
  Since $\alpha_n \neq \alpha$, $\alpha_{n+1} \neq \alpha$ and $\alpha \neq \last(u)$ for any basic term $u$, it is clear from the saturation conditions that also  ${\alpha_n}_\alpha^\kappa \leadsto {\alpha_{n+1}}_\alpha^\kappa$. For the second inequality the proof is analogous.
\end{proof}

\noindent
{\em Proof of \Cref{lemma:finitesaturation}}. 
It suffices to prove that the saturation of $\{ \alpha \}$ is finite for any sample $\alpha$, i.e., that $\mu(\alpha)$ is finite. Clearly, $\mu(\suc(\alpha)) \leq  1 + \mu(\alpha) $ and $\mu(\pre(\alpha)) \leq 1 + \mu(\alpha)$. So, by \Cref{lemma:saturation-inequality}, it suffices to prove that $\mu(\s{t}{\kappa})$ is finite for every term $t$ and time variable $\kappa$, proceeding by induction on $t$. If $t \in \TermV \cup \{ \id, \bot \}$, then $\{ \s{t}{\kappa} \}^\leadsto = \{ \s{t}{\kappa}, \kappa, \s{t}{\last(t)}, \last(t) \}$, so $\mu(\s{t}{\kappa}) = 4$. 

If $t = a_1\cdots a_n$, where $a_1,\ldots,a_n$ are terms that are not products, then by the saturation conditions,
\[
\{  \s{t}{\kappa} \}^\leadsto = \{\s{t}{\kappa},\s{t}{\last(t)})\} \cup  \bigcup_{i=1}^n \bigcup_{ \alpha \in \{\kappa,\last(t) \} } \{ \s{a_1\cdots a_i}{\s{a_{i+1}\cdots a_n}{\alpha}} \}^\leadsto.
\]
So, by \Cref{lemma:saturation-inequality}, 
\[
\mu(\s{t}{\kappa}) \leq 2  + 2 \left( \sum_{i=1}^n    \mu(\s{a_1\cdots a_i}{\kappa}) +  \mu(\s{a_{i+1}\cdots a_n}{\kappa})\right) 
\]
and, by the induction hypothesis, the right-hand-side is finite.

If $t= \ro{u}$, then, by the saturation conditions, 
\[
\{  \s{t}{\kappa} \}^\leadsto  = \{  \s{t}{\kappa}, \s{t}{\last(t)} \}  \cup \{  \s{u}{\kappa} \}^\leadsto \cup \{  \s{u}{\last(t)} \}^\leadsto.
\]
So we get 
\(
\mu(\s{t}{\kappa}) \leq 2 +  2\mu(\s{u}{\kappa})
\)
and, by the induction hypothesis, the right-hand-side is finite.

If $t = \rr{u}$, then clearly  
\[
 \mu( \s{t}{\kappa}) = \mu( \s{u}{\suc(\s{t}{\kappa})}, \s{u}{\s{t}{\kappa}},\s{u}{\suc(\s{t}{\last(t)})}, \s{u}{\s{t}{\last(t)}})
 \]
 and, by applying \Cref{lemma:saturation-inequality} for $\alpha = \s{t}{\kappa}$, $\gamma_1 = \s{u}{\suc(\s{t}{\last(t)})}$, and $\gamma_2 =  \s{u}{\s{t}{\last(t)}}$, 
 \[
 \mu( \s{t}{\kappa}) \leq \mu( \s{u}{\suc(\kappa)}, \s{u}{\kappa},\gamma_1,\gamma_2)  + \mu_\alpha( \s{u}{\suc(\alpha)}, \s{u}{\alpha},\gamma_1,\gamma_2).
\]
But applying  \Cref{lemma:saturation-inequality} again  for  $\alpha' = \s{t}{\last(t)}$ with $\gamma_1' = \s{u}{\suc(\kappa)}$, $\gamma_2' = \s{u}{\kappa}$, and a new time variable $\kappa'$, 
\begin{align*}
\mu( \s{u}{\suc(\kappa)}, \s{u}{\kappa},\gamma_1,\gamma_2)  &\leq \mu(\s{u}{\suc(\kappa')}, \s{u}{\kappa'}, \gamma_1',\gamma_2')  + \mu_{\alpha'}( \s{u}{\suc(\alpha')}, \s{u}{\alpha'},\gamma_1',\gamma_2') \\
&\leq 2\mu( \s{u}{\suc(\kappa)}) +2\mu( \s{u}{\kappa}) + \mu_{\alpha'}( \s{u}{\suc(\alpha')}, \s{u}{\alpha'},\gamma_1',\gamma_2').
\end{align*}
In summary,
\[
\mu( \s{t}{\kappa})  \leq 2\mu( \s{u}{\suc(\kappa)}) +2\mu( \s{u}{\kappa}) + \mu_{\alpha'}( \s{u}{\suc(\alpha')}, \s{u}{\alpha'},\gamma_1',\gamma_2')  +  \mu_\alpha( \s{u}{\suc(\alpha)}, \s{u}{\alpha},\gamma_1,\gamma_2).
\]
By the induction hypothesis, the sum  $2\mu( \s{u}{\suc(\kappa)}) +2\mu( \s{u}{\kappa})$ is finite. But also
\begin{align*}
M_{ \s{t}{\kappa}} ( \s{u}{\suc(\s{t}{\kappa})}, \s{u}{\s{t}{\kappa}},\s{u}{\suc(\s{t}{\last(t)})}, \s{u}{\s{t}{\last(t)}}) & = \{\s{t}{\kappa}, \kappa  \} \\
M_{ \s{t}{\last(t)}} (\s{u}{\suc(\s{t}{\last(t)})}, \s{u}{\s{t}{\last(t)}}, \s{u}{\suc(\kappa)}, \s{u}{\kappa}) & = \{\s{t}{\last(t)}, \last(t)\}.
\end{align*}
So $\mu(\s{t}{\kappa}) $ is finite. 

The case where $t = \rl{u}$ is analogous to the case where $t = \rr{u}$.  \qed

\medskip
Note that this proof yields a rough upper-bound $\mu(\s{t}{\kappa}) \leq (6\cdot c(t))^{c(t)}$, where $c(t)$ is the complexity of the term $t$.


\subsection{Proof of~\Cref{proposition:valuation-to-diagram}}\label{section:proof_delta_is_a_diagram}

To conclude the proof of~\Cref{proposition:valuation-to-diagram}, it remains to prove that $\delta$ is a diagram, i.e., that $\delta$ satisfies conditions~\Cref{c:mon}-\Cref{c:l-last-value-finite}. For convenience, we assume without further mention that all samples used are in $\Delta$, and write $\sem{t}$ for $\sem{t}_\theta$.

 \begin{newlist}
 
 \item[\Cref{c:mon}] 
If $\delta(\alpha) \leq \delta(\beta)$, then, by the definition of $\delta$ and the fact that time warps are monotonic, $\delta(\s{t}{\alpha}) = \sem{t}(\delta(\alpha)) \leq \sem{t}(\delta(\beta)) = \delta(\s{t}{\beta})$.
 
 \item[\Cref{c:zero}] 
 If $\delta(\alpha) = 0$, then $\delta(\s{t}{\alpha}) = \sem{t}(\delta(\alpha)) = 0$.

 \item[\Cref{c:pre}] 
 By the definition of $\delta$.
 
 \item[\Cref{c:suc}] 
 By the definition of $\delta$.
 
 \item[\Cref{c:last}] 
 By the definition of $\delta$, 
  \[
 \delta(\last(t)) = \lastw(\sem{t}) = \min\{ n\in \som \mid \sem{t}(n) = \sem{t}(\om) \}.
 \]
 So clearly, for each $k\in \som$,
 \[
 \lastw(\sem{t}) \leq k \iff \sem{t}(\lastw(\sem{t})) = \sem{t}(k) .
 \]
 Hence, for all $\s{t}{\alpha} \in \Delta$,
 \[
 \delta(\last(t)) \leq \delta(\alpha) \iff \delta(\s{t}{\last(t)}) = \delta(\s{t}{\alpha}).
 \]

 \item[\Cref{c:last2}] 
If $\lastw(\sem{t}) = \delta(\last(t)) = \om$, then $\sem{t}(n) < \sem{t}(\om)$ for all $n<\om$, and $\delta(\s{t}{\last(t)}) = \sem{t}(\om) = \bigvee_{n<\om}\sem{t}(n) =\om$.

 \item[\Cref{c:id}] 
$\delta(\s{\id}{\alpha}) =\sem{\id}(\delta(\alpha))  = \delta(\alpha)$.
 
 \item[\Cref{c:bot}] 
$\delta(\last(\bot)) = \lastw(\sem{\bot}) =0$.
 
 \item[\Cref{c:prod}] 
$\delta(\s{tu}{\alpha}) = \sem{tu}(\delta(\alpha)) = \sem{t}(\sem{u}(\delta(\alpha))) = \delta(\s{t}{\s{u}{\alpha}})$.

 \item[\Cref{c:last-prod1}] 
If $\delta(\last(tu)) = \om$, then $\lastw(\sem{t}\sem{u}) = \om$ and, by \Cref{lemma:last-properties}, $\delta(\last(t)) = \lastw(\sem{t}) = \om$ and $\delta(\last(u)) =\lastw(\sem{u}) = \om$.

 \item[\Cref{c:o-values}] 
By the definition of $\delta$, we have $\delta(\s{\ro{t}}{\alpha}) = \ro{\sem{t}}(\delta(\alpha))$. Moreover, by \Cref{lemma:op-o}, we have $\delta(\s{\ro{t}}{\alpha}) = \ro{\sem{t}}(\delta(\alpha)) = 0$ or $\delta(\s{\ro{t}}{\alpha}) =\ro{\sem{t}}(\delta(\alpha)) = \om$.
 
 \item[\Cref{c:o-inf}] 
If $\delta(\alpha) <\om$, then, by \Cref{lemma:op-o},
  \[
  \delta(\s{\ro{t}}{\alpha}) = \ro{\sem{t}}(\delta(\alpha)) = \om \iff  \delta(\s{t}{\alpha}) = \sem{t}(\delta(\alpha)) = \om.
  \]

 \item[\Cref{c:o-last-finite}] 
$\delta(\last(\ro{t})) =\lastw(\ro{ \sem{t}}) <\om$, by \Cref{lemma:op-o}.
  
 \item[\Cref{c:o-last-value-finite}]
  Suppose that $\ro{\sem{t}}(\lastw(\ro{\sem{t}})) = \delta(\s{\ro{t}}{\last(\ro{t})}) <\om$. Then, since $\ro{\sem{t}}(\lastw(\ro{\sem{t}}))  = \ro{\sem{t}}(\om)$, by \Cref{lemma:op-o}, we get $\sem{t}(k) <\om$ for all $k<\om$. So in particular for all $\delta(\alpha) <\om$, we have $\delta(\s{t}{\alpha}) = \sem{t}(\delta(\alpha)) <\om$.
 
 \item[\Cref{c:r-lower}] 
 $\delta(\s{t}{\s{\rr{t}}{\alpha}}) = \sem{t}(\rr{\sem{t}}(\delta(\alpha))) \leq \delta(\alpha)$, by \Cref{lemma:op-r}.
 
 \item[\Cref{c:r-finite}] 
If $0<\delta(\alpha) <\om$ and $\rr{\sem{t}}(\delta(\alpha)) = \delta(\s{\rr{t}}{\alpha}) <\om$, then $\delta(\alpha) < \sem{t}(\rr{\sem{t}}(\delta(\alpha)) + 1) = \delta(\s{t}{\suc(\s{\rr{t}}{\alpha})})$, by \Cref{lemma:op-r}.

 \item[\Cref{c:r-last-inf}] 
If $\lastw(\rr{\sem{t}}) = \delta(\last(\rr{t})) = \om$, then $\delta(\last(t)) = \lastw(\sem{t}) = \om$, by \Cref{lemma:last-properties}.

 \item[\Cref{c:r-last-value-finite}] 
If $\rr{\sem{t}}(\lastw(\rr{\sem{t}})) = \delta(\s{\rr{t}}{\last(\rr{t})}) <\om$, then $ \rr{\sem{t}}(\om) = \rr{\sem{t}}(\lastw(\rr{\sem{t}}))<\om$ and $\delta(\s{t}{\suc(\s{\rr{t}}{\last(\rr{t})})}) = \sem{t}(\rr{\sem{t}}(\om) + 1) = \om$, by \Cref{lemma:op-r}.
 
 \item[\Cref{c:l-finite1}] 
If $\rl{\sem{t}}(\delta(\alpha)) = \delta(\s{\rl{t}}{\alpha}) <\om$, then either $\delta(\alpha) = 0$ and $\delta(\s{t}{\s{\rl{t}}{\alpha}}) = \sem{t}(\rl{\sem{t}}(0)) = 0$, or $0<\delta(\alpha) <\om$ and $\delta(\alpha) \leq \sem{t}(\rl{\sem{t}}(\delta(\alpha))) = \delta(\s{t}{\s{\rl{t}}{\alpha}})$, by \Cref{lemma:op-l}.

 \item[\Cref{c:l-finite2}] 
If $0<\delta(\alpha) <\om$ and  $\rl{\sem{t}}(\delta(\alpha)) = \delta(\s{\rl{t}}{\alpha}) <\om$, then $\delta(\s{t}{\pre(\s{\rl{t}}{\alpha})}) = \sem{t}(\rl{\sem{t}}(\delta(\alpha)) -1) < \delta(\alpha)$, by \Cref{lemma:op-l}.
 
 \item[\Cref{c:l-inf}] 
If $\delta(\alpha) <\om$  and $\rl{\sem{t}}(\delta(\alpha)) = \delta(\s{\rl{t}}{\alpha}) = \om$, then $\delta(\alpha) > 0$ and $\delta(\s{t}{\s{\rl{t}}{\alpha}} = \sem{t}(\om)  <\delta(\alpha)$, by \Cref{lemma:op-l}.

 \item[\Cref{c:l-last-inf}] 
If $\lastw(\rl{\sem{t}}) = \delta(\last(\rl{t})) = \om$, then $\delta(\last(t)) = \lastw(\sem{t}) = \om$, by \Cref{lemma:last-properties}.

 \item[\Cref{c:l-last-value-finite}] 
If $\rl{\sem{t}}(\lastw(\rl{\sem{t}})) = \delta(\s{\rl{t}}{\last(\rl{t})}) <\om$, then $\rl{\sem{t}}(\om) =  \rl{\sem{t}}(\lastw(\rl{\sem{t}})) <\om$ and 
 $\delta(\s{t}{\s{\rl{t}}{\last(\rl{t})}}) = \sem{t}(\rl{\sem{t}}(\om) ) = \om$, by \Cref{lemma:op-l}.

\end{newlist}


\subsection{Proof of~\Cref{lemma:fundamental}}\label{section:fundamental}

Recall that the proof of~\Cref{lemma:fundamental} proceeds by induction on $t$ and that the case~$t = x$ follows by assumption. The other cases are direct consequences of the following lemmas and the induction hypothesis.

\begin{lem} \label{lemma:fundamental-product}
If~$f_1$ strongly extends~$\diag{t_1}{\delta}$ and~$f_2$ strongly extends~$\diag{t_2}{\delta}$, then~$f_1 f_2$ strongly extends~$\diag{t_1 t_2}{\delta}$.
\end{lem}
\begin{proof}
Suppose that $f_1$ strongly extends~$\diag{t_1}{\delta}$ and~$f_2$ strongly extends~$\diag{t_2}{\delta}$. Then for all~$\s{t_1 t_2}{\alpha} \in \Delta$, 
  \begin{align*}
    f_1 f_2 (\delta(\alpha))
    & =
    f_1 (f_2(\delta(\alpha)))
    & \mbox{(by definition)}
    \\
    & =
    f_1(\delta(\s{t_2}{\alpha}))
    & \mbox{(since $f_2$ extends~$\diag{t_2}{\delta}$)}
    \\
    & =
    \delta(\s{t_1}{\s{t_2}{\alpha}})
    & \mbox{(since $f_1$ extends~$\diag{t_1}{\delta}$)}
    \\
    & =
    \delta(\s{t_1 t_2}{\alpha})
    & \mbox{(by \cref{c:prod})}.
  \end{align*}
 So~$f_1 f_2$ extends~$\diag{t_1 t_2}{\delta}$, and it remains to show that the extension is strong. We can assume that $\diag{t_1 t_2}{\delta}$ is non-empty, since otherwise there is nothing to prove. Suppose that~$\delta(\last(t_1 t_2)) = \om$. Then $\delta(\last(t_1)) = \delta(\last(t_2)) = \om$, by~\cref{c:last-prod1}, and, since $f_1$ and $f_2$ strongly extend $\diag{t_1}{\delta}$ and $\diag{t_2}{\delta}$, respectively, also $\lastw(f_1) = \lastw(f_2) = \om$. Hence $\lastw(f_1 f_2) = \om$, by~\Cref{lemma:last-properties}.
\end{proof}

\begin{lem} \label{lemma:fundamental-o}
 If~$f$ strongly extends~$\diag{t}{\delta}$, then~$\ro{f}$ strongly extends~$\diag{\ro{t}}{\delta}$.
\end{lem}

\begin{proof}
Suppose that $f$ strongly extends~$\diag{t}{\delta}$ and consider any~$\s{\ro{t}}{\alpha} \in \Delta$. We prove that~$\ro{f}(\delta(\alpha)) = \delta(\s{\ro{t}}{\alpha})$.  Suppose first that $\delta(\alpha) <\om$. We reason by cases for~$\delta(\s{\ro{t}}{\alpha})$.
   
\begin{enumerate}

\item[\rm (i)] $\delta(\s{\ro{t}}{\alpha}) = \om$. Then~$\delta(\s{t}{\alpha}) = \om$, by~\cref{c:o-inf}, and, since~$f$ extends~$\diag{t}{\delta}$, also~$f(\delta(\alpha)) = \om$. Hence $\ro{f}(\delta(\alpha)) = \om$, by~\Cref{lemma:op-o}.

\item[\rm (ii)] $\delta(\s{\ro{t}}{\alpha})<\om$. Then $\delta(\s{\ro{t}}{\alpha}) = 0$, by~\cref{c:o-values}, and hence~$\delta(\s{t}{\alpha})<\om$,  
by~\cref{c:o-inf}. Since~$f$ extends~$\diag{t}{\delta}$, also~$f(\delta(\alpha)) <\om$. Hence~$\ro{f}(\delta(\alpha)) = 0$, by~\Cref{lemma:op-o}.

\end{enumerate}

\noindent
Now suppose that~$\delta(\alpha) = \om$. Then $\delta(\s{\ro{t}}{\alpha}) = \delta(\s{\ro{t}}{\last(\ro{t})})$, by~\cref{c:last}, and $\delta(\last(\ro{t})) <\om$, by~\cref{c:o-last-finite}. As in the previous cases,~$\ro{f}(\delta(\last(\ro{t}))) = \delta(\s{\ro{t}}{\last(\ro{t})})$, recalling that by~\cref{c:o-values}, either $\delta(\s{\ro{t}}{\alpha}) = \om$ or $\delta(\s{\ro{t}}{\alpha}) = 0$.

\begin{enumerate}

\item $\delta(\s{\ro{t}}{\last(\ro{t})}) =\delta(\s{\ro{t}}{\alpha}) = \om$. Then~$\ro{f}(\delta(\last(\ro{t}))) = \om$, and~$\ro{f}(\om) = \om$.

\item $\delta(\s{\ro{t}}{\last(\ro{t})}) =\delta(\s{\ro{t}}{\alpha}) = 0$. Then there are two cases. If $\delta(\s{t}{\alpha}) <\om$, then, since $f$ extends $\diag{t}{\delta}$, we have $f(\om) <\om$ and $\ro{f}(\om) = 0$, by \Cref{lemma:op-o}. Otherwise, $\delta(\s{t}{\alpha}) = \om$. In this case, $\delta(\s{t}{\beta})<\om$ for all $\delta(\beta) <\om$ with $\s{t}{\beta}\in \Delta$,  by \cref{c:o-last-value-finite}, so~$\delta(\last(t)) = \om$. Hence, since $f$ strongly extends $\diag{t}{\delta}$, we have $\lastw(f) = \om$ and~$\ro{f}(\om) = 0$, by \Cref{lemma:op-o}.

\end{enumerate}

That $\ro{f}$ strongly extends~$\diag{\ro{t}}{\delta}$ is clear, since $\delta(\last(\ro{t})) <\om$, by \cref{c:o-last-finite}. 
\end{proof}

\begin{lem}\label{lemma:fundamental-r} 
If~$f$ strongly extends~$\diag{t}{\delta}$, then~$\rr{f}$ strongly extends~$\diag{\rr{t}}{\delta}$.
\end{lem}
\begin{proof}
Let $\s{\rr{t}}{\alpha} \in \Delta$. Note first that, by \cref{c:zero}, if $\delta(\alpha) = 0$, then $\delta(\s{\rr{t}}{\alpha}) = 0=\rr{f}(0)$. Hence assume  that $\delta(\alpha)> 0$. Suppose first that $\delta(\alpha) <\om$. We reason by cases for $\delta(\s{\rr{t}}{\alpha})$.
  
  \begin{enumerate}
  \item $\delta(\s{\rr{t}}{\alpha})  <\om$. Then  $\delta(\s{t}{\s{\rr{t}}{\alpha}})\leq \delta(\alpha) < \delta(\s{t}{\suc(\s{\rr{t}}{\alpha})})$, by \cref{c:r-lower} and \cref{c:r-finite}.  Since $f$ extends~$\diag{t}{\delta}$, we have $f(\delta(\s{\rr{t}}{\alpha})) =\delta(\s{t}{\s{\rr{t}}{\alpha}}) \le \delta(\alpha) < \delta(\s{t}{\suc(\s{\rr{t}}{\alpha})}) = f(\delta(\s{\suc(\rr{t}}{\alpha})))$, and, by~\cref{c:suc}, also $f(\delta(\s{\suc(\rr{t}}{\alpha})))= f(\delta(\s{\rr{t}}{\alpha})+1) $.  So $\delta(\s{\rr{t}}{\alpha}) =\rr{f}(\delta(\alpha))$, by \Cref{lemma:op-r}.
   
   \item $\delta(\s{\rr{t}}{\alpha}) = \om$. Then, since $\delta(\s{t}{\s{\rr{t}}{\alpha}}) \leq \delta(\alpha)$, by \cref{c:r-lower},  and $f$ extends~$\diag{t}{\delta}$, we have $f(\om) =  f(\delta(\s{\rr{t}}{\alpha})) \leq \delta(\alpha) <\om$. Hence $\rr{f}(\delta(\alpha)) = \om$, by \Cref{lemma:op-r}.
  \end{enumerate}
  
 \noindent 
 Now suppose that $\delta(\alpha) = \om$ and hence $\delta(\s{\rr{t}}{\alpha}) =\delta(\s{\rr{t}}{\last(\rr{t})})$. We reason by cases for $\delta(\s{\rr{t}}{\alpha})$.
  
  \begin{enumerate}
  \item $\delta(\s{\rr{t}}{\last(\rr{t})}) =\delta(\s{\rr{t}}{\alpha})  <\om$. Then  $\delta(\last(\rr{t})) <\om$ and $\delta(\s{t}{\suc(\s{\rr{t}}{\last(\rr{t})})}) = \om$, by \cref{c:last2} and \cref{c:r-last-value-finite}. So, using the previous cases, \cref{c:suc}, and the fact that $f$ extends~$\diag{t}{\delta}$, we get $\rr{f}(\delta(\last(\rr{t}))) = \delta(\s{\rr{t}}{\alpha})$ and $f(\delta(\s{\rr{t}}{\alpha})+1) = \om$. Hence $\rr{f}(\delta(\alpha)) = \delta(\s{\rr{t}}{\alpha})$, by \Cref{lemma:op-r}.
  
  \item $\delta(\s{\rr{t}}{\last(\rr{t})}) =\delta(\s{\rr{t}}{\alpha})  =\om$. Then there are two cases. If $\delta(\last(\rr{t})) <\om$, then, using the previous cases, $\rr{f}(\delta(\last(\rr{t}))) = \om$, and hence  $\rr{f}(\om) = \om$ by \Cref{lemma:op-r}. Otherwise $\delta(\last(\rr{t})) = \om$. Then $\delta(\last(t)) = \om$, by \cref{c:r-last-inf}, and since $f$ strongly extends~$\diag{t}{\delta}$, we have $\lastw(f) = \om$. Hence $\rr{f}(\om) = \om$, by \Cref{lemma:op-r}.
  \end{enumerate}
  It remains to show that the extension is strong. Again we can assume that $\diag{\rr{t}}{\delta}$ is non-empty. Suppose that $\delta(\last(\rr{t}))  =\om$. Then $\delta(\last(t)) = \om$, by \cref{c:r-last-inf}, and, since $f$ strongly extends~$\diag{t}{\delta}$, also $\lastw(f) = \om$.  Hence $\lastw(\rr{f}) = \om$, by \Cref{lemma:last-properties}.
\end{proof}
\begin{lem}\label{lemma:fundamental-l} 
If~$f$ strongly extends~$\diag{t}{\delta}$, then~$\rl{f}$ strongly extends~$\diag{\rl{t}}{\delta}$.
\end{lem}
\begin{proof}
Let $\s{\rl{t}}{\alpha} \in \Delta$. Note first that, by \cref{c:zero}, if $\delta(\alpha) = 0$, then $\delta(\s{\rl{t}}{\alpha}) = 0=\rl{f}(0)$. Hence assume that $\delta(\alpha)> 0$. Suppose first that $\delta(\alpha) <\om$. We reason by cases for $\delta(\s{\rl{t}}{\alpha})$.
  
  \begin{enumerate}
  \item $\delta(\s{\rl{t}}{\alpha})  <\om$. Then $\delta(\s{t}{\pre(\s{\rl{t}}{\alpha})}) <\delta(\alpha) \leq \delta(\s{t}{\s{\rl{t}}{\alpha}})$, by \cref{c:l-finite1} and \cref{c:l-finite2}. Since $f$ extends~$\diag{t}{\delta}$, we have $f(\delta({\pre(\s{\rl{t}}{\alpha})}))=\delta(\s{t}{\pre(\s{\rl{t}}{\alpha})})<\delta(\alpha)\le\delta(\s{t}{\s{\rl{t}}{\alpha}})= f(\delta(\s{\rl{t}}{\alpha}))$. But also $f(\delta({\pre(\s{\rl{t}}{\alpha})})) = f(\delta(\s{\rl{t}}{\alpha})-1)$, by \cref{c:pre}, noting that  $0<\delta(\s{\rl{t}}{\alpha})$, since $f(0) = 0 < \delta(\alpha)$. Hence $\delta(\s{\rl{t}}{\alpha}) =\rl{f}(\delta(\alpha))$, by \Cref{lemma:op-l}. 
   
   \item $\delta(\s{\rl{t}}{\alpha}) = \om$. Then $\delta(\s{t}{\s{\rl{t}}{\alpha}}) < \delta(\alpha)$, by \cref{c:l-inf}, and, since $f$ extends~$\diag{t}{\delta}$, also $f(\om) =  f(\delta(\s{\rl{t}}{\alpha})) < \delta(\alpha)$. Hence $\rl{f}(\delta(\alpha)) = \om$, by \Cref{lemma:op-l}.
  \end{enumerate}
  
  \noindent 
  Suppose now that $\delta(\alpha) = \om$ and hence $\delta(\s{\rl{t}}{\alpha}) =\delta(\s{\rl{t}}{\last(\rl{t})})$. We reason by cases on $\delta(\s{\rl{t}}{\alpha})$.
  
  \begin{enumerate}
  \item $\delta(\s{\rl{t}}{\last(\rl{t})}) =\delta(\s{\rl{t}}{\alpha})  <\om$. Then $\delta(\last(\rl{t})) <\om$ and $\delta(\s{t}{\s{\rl{t}}{\last(\rl{t})}}) = \om$, by \cref{c:last2} and \cref{c:l-last-value-finite}. So, by the previous cases and the fact that $f$ extends~$\diag{t}{\delta}$, we have $\rl{f}(\delta(\last(\rl{t}))) = \delta(\s{\rl{t}}{\alpha})$ and $f(\delta(\s{\rl{t}}{\alpha})) = \om$. Hence $\rl{f}(\delta(\alpha)) = \delta(\s{\rl{t}}{\alpha})$, by \Cref{lemma:op-l}.
  
  \item $\delta(\s{\rl{t}}{\last(\rl{t})}) =\delta(\s{\rl{t}}{\alpha})  =\om$. Then there are two cases. If $\delta(\last(\rl{t})) <\om$, then by the previous cases, $\rl{f}(\delta(\last(\rl{t}))) = \om$, and $\rl{f}(\om) = \om$ by \Cref{lemma:op-l}. If $\delta(\last(\rl{t})) = \om$, then $\delta(\last(t)) = \om$, by \cref{c:l-last-inf}, and, since $f$ strongly extends~$\diag{t}{\delta}$, also $\lastw(f) = \om$ and $\rl{f}(\om) = \om$, by \Cref{lemma:op-l}.
  \end{enumerate}
  It remains to show that the extension is strong. Again we can assume that $\diag{\rl{t}}{\delta}$ is non-empty. Suppose that $\delta(\last(\rl{t}))  =\om$. Then, by \cref{c:l-last-inf}, we get $\delta(\last(t)) = \om$. So, since $f$ strongly extends~$\diag{t}{\delta}$, also $\lastw(f) = \om$. Hence $\lastw(\rl{f}) = \om$, by \Cref{lemma:last-properties}. 
\end{proof}

\begin{lem}
  \label{lemma:fundamental-id}
  The time warp~$\id$ strongly extends~$\diag{\id}{\delta}$.
\end{lem}

\begin{proof}
  The extension property follows from~\cref{c:id}; the fact that it is strong follows from the fact that~$\lastw(\id) = \om$. 
\end{proof}

\begin{lem}
  \label{lemma:fundamental-bot}
  The time warp~$\bot$ strongly extends~$\diag{\bot}{\delta}$.
\end{lem}

\begin{proof}
  The extension property follows from~\cref{c:last}  and~\cref{c:bot}; the fact that it is strong is immediate by~\cref{c:bot}.
  \end{proof}
  

\end{document}